\documentclass[12pt]{amsart}
\usepackage[top=4.2cm, bottom=4cm, left=2.4cm, right=2.4cm]{geometry}
\usepackage[utf8]{inputenc}
\usepackage[USenglish]{babel}
\usepackage[T1]{fontenc} 
\usepackage{mathrsfs}
\usepackage{mathtools}
\usepackage{amsmath}
\usepackage{amssymb,epsfig}
\usepackage{caption}
\usepackage{subcaption}
\usepackage{comment}

\usepackage{amsthm}
\usepackage[absolute]{textpos}
\usepackage{mathtools,emptypage}

\usepackage[bookmarks=true]{hyperref}
\usepackage{xcolor}
\hypersetup{
    colorlinks,
    linkcolor={red!50!black},
    citecolor={blue!50!black},
    urlcolor={blue!80!black},
}

\usepackage{todonotes}

\makeatletter
\newtheorem*{rep@theorem}{\rep@title}
\newcommand{\newreptheorem}[2]{%
\newenvironment{rep#1}[1]{%
 \def\rep@title{#2 \ref{##1}}%
 \begin{rep@theorem}}%
 {\end{rep@theorem}}}
\makeatother

\newreptheorem{theorem}{Theorem}

\mathtoolsset{showonlyrefs} % uncomment to tag only the referred equations (post-production) 

\newcommand{\N}{\mathbb{N}}

\newcommand{\R}{\mathbb{R}}

\newcommand{\sfd}{{\sf d}}

\newcommand{\supp}{\mathop{\rm supp}\nolimits}   
  
\renewcommand{\d}{{\mathrm d}}

\newcommand{\restr}[1]{\lower3pt\hbox{$|_{#1}$}}

\newcommand{\limi}{\varliminf}
\newcommand{\lims}{\varlimsup}
                          
\newcommand{\X}{{\rm X}}
\newcommand{\Y}{{\rm Y}}

\newcommand{\lip}{{\rm lip}}
\newcommand{\Lip}{{\rm Lip}}
\newcommand{\loc}{{\rm loc}}

\newcommand{\mm}{\mathfrak m}

\newcommand{\UG}{{\rm UG}}

\newcommand{\width}{{\rm width}}
\newcommand{\dwidth}{{\rm disc\text{-}width}}
\renewcommand{\SS}{{\rm SS}}
\newcommand{\dSS}{{\rm disc\text{-}SS}}

\newcommand{\SR}{{\rm SR}}
\newcommand{\dSR}{{\rm disc\text{-}SR}}
\newcommand{\V}{{\rm V}}
\newcommand{\pos}{{\rm pos}}
\newcommand{\sfc}{{\sf c}}

\makeatletter
\newcommand{\mytag}[2]{%
  \text{#1}%
  \@bsphack
  \begingroup
    \@onelevel@sanitize\@currentlabelname
    \edef\@currentlabelname{%
      \expandafter\strip@period\@currentlabelname\relax.\relax\@@@%
    }%
    \protected@write\@auxout{}{%
      \string\newlabel{#2}{%
        {\color{black}#1}%
        {\thepage}%
        {\@currentlabelname}%
        {\@currentHref}{}%
      }%
    }%
  \endgroup
  \@esphack
}
\makeatother

\def\Xint#1{\mathchoice
{\XXint\displaystyle\textstyle{#1}}%
{\XXint\textstyle\scriptstyle{#1}}%
{\XXint\scriptstyle\scriptscriptstyle{#1}}%
{\XXint\scriptscriptstyle\scriptscriptstyle{#1}}%
\!\int}
\def\XXint#1#2#3{{\setbox0=\hbox{$#1{#2#3}{\int}$ }
\vcenter{\hbox{$#2#3$ }}\kern-.6\wd0}}

\def\dashint{\Xint-}

\allowdisplaybreaks

\newtheorem{theorem}{Theorem}[section]

\newtheorem{corollary}[theorem]{Corollary}
\newtheorem{lemma}[theorem]{Lemma}
\newtheorem{proposition}[theorem]{Proposition}
\theoremstyle{definition}
\newtheorem{definition}[theorem]{Definition}

\newcounter{Counter}

\newtheorem{remark}[theorem]{Remark}

\newcommand{\Int}{\textup{int}}
\newcommand{\Mod}{\textup{Mod}}
\newcommand{\codH}[1]{\mathcal{H}^{{\rm cod-}{#1}}}
\title[Poincar\'{e} inequality and Minkowski content of separating sets]{A geometric approach to Poincar\'{e} inequality and Minkowski content of separating sets}

\author[Emanuele Caputo]{Emanuele Caputo}\address[Emanuele Caputo]{Mathematics Institute, Zeeman Building, University of Warwick, Coventry, CV4 7AL, United Kingdom}\email{emanuele.e.caputo@jyu.fi}
\author[Nicola Cavallucci]{Nicola Cavallucci}\address[Nicola Cavallucci]{
Institute of Mathematics, EPFL, Station 8, 1015 Lausanne, Switzerland}\email{n.cavallucci23@gmail.com}


\begin{document}

\maketitle

\begin{abstract}
    The goal of this paper is to continue the study of the relation between the Poincar\'{e} inequality and the lower bounds of Minkowski content of separating sets, initiated in our previous work \cite{CaputoCavallucci2024}. A new shorter proof is provided. An intermediate tool is the study of the lower bound of another geometric quantity, called separating ratio. The main novelty is the description of the relation between the infima of the separating ratio and the Minkowski content of separating sets. We prove a quantitative comparison between the two infima in the local quasigeodesic case and equality in the local geodesic one. 
    No Poincar\'{e} assumption is needed to prove it. The main tool employed in the proof is a new function, called the position function, which allows in a certain sense to fibrate a set in boundaries of separating sets. We also extend the proof to measure graphs, where due to the combinatorial nature of the problem, the approach is more intuitive. In the appendix, we revise some classical characterizations of the $p$-Poincar\'{e} inequality, by proving along the way equivalence with a notion of $p$-pencil that extends naturally the definition for $p=1$.
\end{abstract}

\tableofcontents

\section{Introduction}
Doubling metric measure spaces satisfying a $p$-Poincaré inequality for some $p\ge 1$ are natural objects to study problems in geometric measure theory and potential theory on metric measure spaces. This was particularly clear after the theory developed in \cite{Cheeger99} and \cite{HK00}. As such, this class has widely spread and is by now very common in works on analysis in metric spaces.
Although the Poincar\'{e} inequality has an analytical formulation, it has been realised that it is in fact a profoundly geometric condition.
Under the standing doubling assumption, it can be characterized in purely geometric terms (see \cite{Kei03}, \cite{KorteLahti2014}, \cite{FasslerOrponen19}, \cite{ErikssonBique2019II}, \cite{DurCarErikBiqueKorteShanmu21}, \cite{CaputoCavallucci2024}). 
In this paper we continue the study of the geometric condition formulated by the authors in \cite{CaputoCavallucci2024} in terms of energy associated to separating sets, for the case $p=1$. In particular, among the several energies studied in \cite{CaputoCavallucci2024}, we focus on the Minkowski content with respect to the measure weighted with the Riesz potential. To be precise we recall the definition of separating sets and of Riesz potential.
\begin{definition}
\label{def:separating_sets}
    Let $(\X,\sfd,\mm)$ be a metric measure space and let $x,y\in \X$. A closed set $\Omega$ is a separating set from $x$ to $y$ if there exists $r>0$ such that $B_r(x) \subseteq \Int(\Omega)$ and $B_r(y) \subseteq \Omega^c$. The class of all separating sets from $x$ to $y$ is denoted by $\SS_{\textup{top}}(x,y)$.
\end{definition}
\begin{definition}
     Let $(\X,\sfd,\mm)$ be a metric measure space.
    Given $x,y\in \X$ and $L\geq 1$, the \emph{$L$-truncated Riesz potential with poles at $x,y$} is
\begin{equation*}
    R_{x,y}^L(z):= \chi_{B_{x,y}^L}(z) \left( \frac{\sfd(x,z)}{\mm(B_{\sfd(x,z)}(x))} + \frac{\sfd(y,z)}{\mm(B_{\sfd(y,z)}(y))}\right),
\end{equation*}
where $B_{x,y}^L = B_{2L\sfd(x,y)}(x) \cup B_{2L\sfd(x,y)}(y)$.
Its associated measure is $\mm_{x,y}^L = R_{x,y}^L \mm$.
\end{definition}
Given a couple of points $x,y \in \X$, we formulate the following condition.
\begin{equation*}
    {\text{(BMC)}}_{x,y}\quad \exists c>0, L\geq 1\text{ such that }(\mm_{x,y}^L)^{+}(\Omega) \ge c\text{ for every }\Omega \in \SS_{\textup{top}}(x,y),
\end{equation*}
where $(\mm_{x,y}^L)^{+}(\Omega)$ denotes the Minkowski content of the set $\Omega$ with respect to the measure $\mm_{x,y}^L$, see Section \ref{subsec-Minkowski} for the details.
The notation ${\text{(BMC)}}_{x,y}$ stands for `big Minkowski content of separating sets at $x,y$'. 
We recall that, according to \cite[Theorem 1.2]{CaputoCavallucci2024}, a doubling metric measure space has a $1$-Poincar\'{e} inequality if and only if ${\text{(BMC)}}_{x,y}$ holds for every $x,y \in \X$ with uniform constants.
This statement is a quantification of the well-known heuristic intuition that the validity of the Poincar\'{e} inequality is obstructed by the presence of bottlenecks in the metric space. The Minkowski content computed via the reference measure weighted with the Riesz potential is the correct energy to measure the size of boundaries of separating sets. 
\vspace{2mm}

\noindent Typically, all the other geometric characterizations of the Poincaré inequality in literature are expressed in terms of properties of curves connecting two points. One of the goals of this paper is to investigate thoroughly the relation between these classical approaches and the (BMC)$_{x,y}$ property. With this in mind we consider a definition similar to the one used in \cite{Sylvester-Gong-21} (originated from previous contributions in \cite{BateLi18} and \cite{ErikssonBique2019}) for which we need an auxiliary quantity. We define the width of a set $A \subseteq \X$ with respect to the points $x,y\in \X$ as 
$$\width_{x,y}(A) := \inf_{\gamma \in \Gamma_{x,y}} \ell(\gamma \cap A),$$
where $\Gamma_{x,y}$ is the set of rectifiable paths connecting $x$ to $y$. The quantity $\width_{x,y}(A)$ measures the width of the set $A$ in the following sense: we consider all the curves (with finite length) connecting $x$ to $y$ and we look at the one whose length inside $A$ is minimal.

\begin{definition}
\label{def:CLLsetconnected}
    Let $C>0$ and $L \geq 1$. We say that $(\X,\sfd,\mm)$ is $(C,L)$ $1$-set-connected at $x,y\in \X$ if
\begin{equation}
    \label{eq:defin_set_connectedness_bounded_intro}
    \width_{x,y}(A)\leq C\mm_{x,y}^L(A) \text{ for all } A\subseteq \X  \text{ Borel}.
\end{equation}
\end{definition}
This definition is inspired by \cite{Sylvester-Gong-21}, however there are two main differences. In \cite{Sylvester-Gong-21} the definition that the authors use can be rephrased as
\begin{equation}
\label{eq:defin_set_connectedness_bounded_maximal_function}
    \width_{x,y}^\Lambda(A)\leq C\sfd(x,y)\left( M_{L\sfd(x,y)}\chi_A(x) + M_{L\sfd(x,y)}\chi_A(y) \right) \text{ for all } A\subseteq \X  \text{ Borel},
\end{equation}
where $M_\rho(f)(x)$ is the Hardy-Littlewood maximal function of a function $f \ge 0$ computed at the point $x$ up to scale $\rho>0$ (see \cite{HK00}) and where in the definition of $\width_{x,y}^\Lambda(A)$ one only considers $\Lambda$-quasigeodesics connecting $x$ and $y$, see Section \ref{sec:set-connectedness} for details. The first difference is in the class of curves used for defining the width. Notice that in \cite{Sylvester-Gong-21}, the finiteness of $\Lambda$ in \eqref{eq:defin_set_connectedness_bounded_maximal_function} is crucial to prove that this condition \emph{for every couple of points} and with uniform constants implies a $p$-Poincaré inequality for all $p>1$ (see \cite[Theorem 2.19 \& Lemma 2.20]{Sylvester-Gong-21}). 
%Another difference is that for their proof they need that condition \eqref{eq:defin_set_connectedness_bounded_maximal_function} holds true for every $x,y\in \X$ with uniform constants.
The second difference is in the right hand sides. They reflect the two equivalent classical pointwise estimates of Heinonen (see \cite[Sec.\ 9]{Hei01}) 
\begin{itemize}
    \item $|u(x)-u(y)| \le C \int \lip u \,\d \mm_{x,y}^L$ for all $u \in {\Lip}(\X)$ and for every $x,y \in\X$,
    \item $|u(x)-u(y)| \le C\,\sfd(x,y)\,(M_{L\sfd(x,y)}(\lip u)(x)+M_{L\sfd(x,y)}(\lip u)(y))$ for all $u \in {\Lip}(\X)$ and for every $x,y \in\X$,
\end{itemize}
whose validity characterizes the $1$-Poincar\'{e} inequality in doubling setting. We recall that $\lip\,u(x)$ denotes the local Lipschitz constant of the function $u$ at $x \in \X$.\\
The first result of this paper shows that the pointwise estimates à la Heinonen for a fixed couple of points $x,y\in\X$, the (BMC)$_{x,y}$ condition and Definition \ref{def:CLLsetconnected} are all equivalent, if $\X$ satisfies mild connectedness assumptions.
\begin{theorem}
\label{theo:main-intro-p=1}
    Let $(\X,\sfd,\mm)$ be a doubling metric measure space which is path connected and locally $\Lambda$-quasiconvex. Let $x,y \in \X$. Then the following conditions are quantitatively equivalent:
    \begin{itemize}
        \item[(i)] there exist $C >0$, $L \ge 1$ such that
        \begin{equation}
        \label{eq:pointwise_Poincaré_Riesz}
            |u(x)-u(y)| \le C \int \lip u \,\d \mm_{x,y}^L\qquad \text{for all }u \in {\Lip}(\X);
        \end{equation}
        \item[(ii)] there exist $C >0$, $L \ge 1$ such that
        \begin{equation}
        \label{eq:pointwise_Poincaré_Riesz_UG}
            |u(x)-u(y)| \le C \int g \,\d \mm_{x,y}^L\qquad \text{for all }u \text{ Borel and all } g \in \UG(u);
        \end{equation}
        \item[(iii)] the space $(\X,\sfd,\mm)$ is $(C,L)$ $1$-set-connected at $x,y$; 
        \item[(iv)] there exist $C,L$ such that \eqref{eq:defin_set_connectedness_bounded_intro} is satisfied by every closed subset $A\subseteq \X$;
        \item[(v)] the space $(\X,\sfd,\mm)$ satisfies \textup{(BMC)}$_{x,y}$.
    \end{itemize}
\end{theorem}
Some remarks are in order. First of all, the equivalence between items (i) and (v) in the statement above is true even dropping the assumption that the metric space is complete (this will be proved in our forthcoming work \cite{CaputoCavallucci2024III}).
Secondly, in Theorem \ref{theo:main-intro-p=1} we just work with \emph{two fixed $x,y$} at the price of assuming local connectivity properties of the metric space. 
%Doing that, assuming \eqref{eq:defin_set_connectedness_bounded_intro} at $x,y$, we get directly the pointwise estimate \eqref{eq:pointwise_Poincaré_Riesz} for $p=1$. \\
By looking at the previous statement for every couple of points, we achieve the next result, which is new and improves the mentioned result of \cite{Sylvester-Gong-21}.
\begin{corollary}
\label{cor:Poincaré_p=1}
    Let $(\X,\sfd,\mm)$ be a doubling metric measure space. Then it satisfies a $1$-Poincaré inequality if and only if there exist $C >0$ and $L \geq 1$ such that
    $$\width_{x,y}(A) \leq C\mm_{x,y}^L(A)$$
    for every $x,y\in\X$ and every $A\subseteq \X$ Borel.
\end{corollary}

Indeed using the same language we introduced above, we can rephrase the results in \cite{Sylvester-Gong-21} for general $p\ge 1$ in terms of our definitions. Their results read as follows: if $\X$ satisfies a $p$-Poincar\'{e} inequality, then $\X$ is $(C,L,\Lambda)$ $p$-set-connected. Conversely, if $\X$ is $(C,L,\Lambda)$ $p$-set-connected, then it satisfies a $q$-Poincar\'{e} inequality for every $q >p$. We refer to Section \ref{sec:set-connectedness} for the definition of $(C,L,\Lambda)$ $p$-set-connectedness and the proof of this statement. Corollary \ref{cor:Poincaré_p=1} states the full equivalence between the two conditions (even with infinite $\Lambda$) in case $p=1$.
%As outlined in \cite{CaputoCavallucci2024}, it is not possible to extend the equivalence with some variant of the (BMC) condition in case $p>1$. 
\vspace{2mm}

The proof of Theorem \ref{theo:main-intro-p=1} is relatively elementary, as it uses classical tools in analysis on metric spaces and the coarea inequality for the Minkowski content. We just want to sketch the proof of (iii) implies (v).
Given a set $\Omega$, we apply (iii) to the sets $B_r(\Omega)\setminus \Omega$ for $r$ sufficiently small getting
$$\width_{x,y}(B_r(\Omega)\setminus \Omega) \leq C\mm_{x,y}^L(B_r(\Omega) \setminus \Omega).$$
It is straightforward to check that $\width_{x,y}(B_r(\Omega)\setminus \Omega) \ge r$ if $\Omega \in \SS_{\rm top}(x,y)$. The conclusion follows by dividing by $\width_{x,y}(B_r(\Omega)\setminus \Omega)$ and taking the limit as $r \to 0$.
\vspace{1mm}

As suggested by the idea above, to compare the $1$-set-connectedness and the Minkowski content computed with respect to $\mm_{x,y}^L$ it is natural to introduce the following notion, which has the same dimensionality of the Minkowski content. The separating ratio between $x$ and $y$ of a set $A \subseteq \X$ is
\begin{equation}
    \label{eq:sep-ratio}
    \SR_{x,y}(A) := \frac{\mm_{x,y}^L(A)}{\width_{x,y}(A)}.
\end{equation}

\noindent With this notation the proof of (iv) implies (v) in Theorem \ref{theo:main-intro-p=1} gives
\begin{equation}
    \inf_{A \subseteq \X\text{ closed}} \SR_{x,y}(A) \leq \inf_{\Omega \in \SS_{\text{top}}(x,y)} (\mm_{x,y}^L)^+(\Omega)
\end{equation}
and Theorem \ref{theo:main-intro-p=1} says that a positive lower bound for one of the two infima is equivalent to a positive lower bound for the other one. However, inspired by the proof given above, it seems possible to approach directly the equivalence between (iv) and (v). Indeed we can quantitatively relate, with explicit constants, the infima of the two functionals. This is the content of the next result.

\begin{theorem}
\label{thm:main_equivalence_minima}
Let $(\X,\sfd,\mm)$ be a doubling, path-connected, locally $\Lambda$-quasiconvex metric measure space and let $x,y\in \X$. Then
\begin{equation}
    \label{eq:equivalence_of_minima_intro}
    \Lambda^{-1}\inf_{\Omega \in \SS_{\textup{top}}(x,y)} (\mm_{x,y}^L)^+(\Omega) \le \inf_{A \subseteq \X\, \textup{closed}} \SR_{x,y}(A)\le \inf_{\Omega \in \SS_{\textup{top}}(x,y)} (\mm_{x,y}^L)^+(\Omega).
\end{equation}
In particular, if $(\X,\sfd)$ is path connected and locally geodesic we have
\begin{equation}
    \label{eq:equality_of_minima_geodesic_case}
    \inf_{A \subseteq \X\, \textup{closed}} \SR_{x,y}(A)= \inf_{\Omega \in \SS_{\textup{top}}(x,y)} (\mm_{x,y}^L)^+(\Omega).
\end{equation}
\end{theorem}

Let us remark that \eqref{eq:equivalence_of_minima_intro} follows by the application of Theorem \ref{theo:main-intro-p=1} as well, with a non explicit control of the constant in the first inequality.
Moreover, the second statement in Theorem \ref{thm:main_equivalence_minima} does not follow from Theorem \ref{theo:main-intro-p=1}.
Besides the discussion about the constants, we believe that the importance of the statement relies on its proof, which is more geometrical rather than analytical, does not use neither the results nor the methods of Theorem \ref{theo:main-intro-p=1} and allows a better understanding of the separating ratio and the Minkowski content. We discuss it in detail in the next subsection.

%We define $\SS(x,y)$ to be the class of subsets $E\subseteq \X$ such that $E\cap \gamma \neq \emptyset$ for all $\gamma \in \Gamma_{x,y}^{\rm rect}$. It is called the class of (pathwise) separating sets from $x$ to $y$.
%
%Let us comment on how to couple Theorem xx with the statement of Theorem xx.
%We can formulate in analogy with the work \cite{CaputoCavallucci2024} the following property:
%\begin{equation}
%    ({BMC}_c)
%\end{equation}
%The subscript $c$ stands for curvewise to distinguish it with the corresponding property formulated in \cite{CaputoCavallucci2024}, equivalent to Poincar\'{e} and formulated in terms of topological sepearating sets ${\rm SS}_{\rm top}(x,y)$.

%In other words the condition $({BMC}_c)$, as a consequence of Theorem xx implies the a priori stronger condition ($1$-BSR) of Theorem RIFERIMENTO, which is equivalent to a $1$-Poincaré inequality. This is the main conceptual reason why for $p=1$ we can use the much more computable condition $(BMC_c)$ to detect the $1$-Poincaré inequality.\\

\subsection{The position function and the proof of Theorem \ref{thm:main_equivalence_minima}}

We start the discussion on the proof of Theorem \ref{thm:main_equivalence_minima} with a toy example. We warn the reader that the sketch of the proof in this example is informal and a bit imprecise. Let us consider the two dimensional Euclidean space and a couple of points, say $x$ and $y$. Let us compute the infimum of the separating ratio between $x$ and $y$. The choice of the dumbbell-shaped domain $D$ as in Figure \ref{fig:dumbbell} as a competitor is not convenient for the minimization of the separating ratio. Indeed, we can reduce the volume of the set (with respect to $\mm_{x,y}^L$) by keeping unaltered its width. This procedure leads to the rectangle $R$ in Figure \ref{fig:dumbbell}. The next step is to chop $R$ in $N$ slimmer rectangles $R_i$ of width equal to the $N$-th fraction of the width of $R$. We then select one of such slimmer rectangles, say $R_1$, with the property that its separating ratio is less than or equal the separating ratio of $R$. This is a consequence of the following formula:
$$\frac{1}{N}\sum_{i=1}^N \SR_{x,y}(R_i) = \SR_{x,y}(R).$$
This formula is trivial in such a case and can be regarded as a discrete version of coarea formula. By iterating this procedure we find slimmer and slimmer rectangles whose separating ratio is smaller and smaller such that they converge to a line $L$. The half-plane with boundary $L$ and containing $x$ is a topological separating set between $x$ and $y$ and its Minkowski content can be estimated from above by the separating ratio of the converging rectangles.

\begin{figure}[h!]
    \centering
    \includegraphics[scale=0.8]{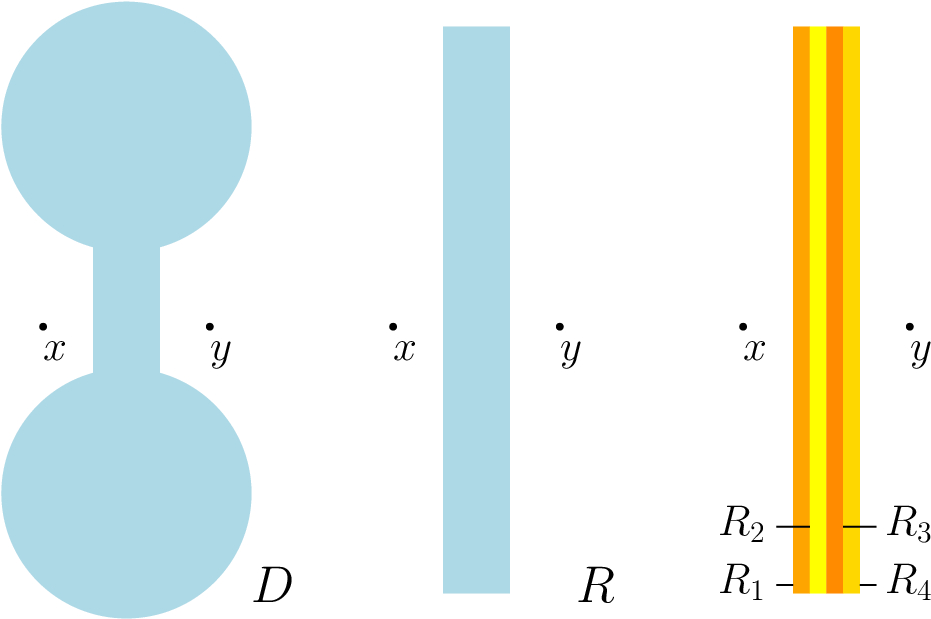}
    \caption{The picture gives an informal explanation of the proof of the Theorem in the the toy example of the two dimensional Euclidean case for a specific choice of $x,y$ and $D$ in the definition of separating ratio.}
    \label{fig:dumbbell}
\end{figure}

\vspace{5mm} 

%{\color{red}
%By keeping the presentation at a very informal level, let us consider as a toy example $\mathbb{R}^2$ endowed with Euclidean distance and Lebesgue measure. In such a case, one wants to reduce $\mm_{x,y}(E)$ by keeping the width $E$ as big as possible.
%So one could consider a very slim rectangle $R$ far from the poles as separating set with edge lenghts $l$ and $\varepsilon$. The measure $\mm_{x,y}(R)$ of such a set is $R_0 l \epsilon$, where $R_{x,y} \simeq R_0$ and $R_0>0$ is a constant that depends only from the distance from the poles. The width of the rectangle is clearly $\epsilon$. Thus $\mm_{x,y}^L(E)/\width_{x,y}(E) \simeq R_0 l$. The last quantity is the Minkowski content of a vertical line in $\mathbb{R}^2$ in the middle of the rectangle with respect to $\mm_{x,y}^L$.
%This type of slim rectangles or tubular neighborhood of sets seems optimal in the computation of the separating sets.
%Indeed, the goal of the paper is to prove that very slim obstacles are enough to compute the separating ratio, namely, by translating into mathematical terms, that the converse inequality in \eqref{eq:separating_ratio_minkowski_above} holds.}

The same ideas can be formalized in the general setting. The main ingredient is a specific function we defined for such a purpose, called the position function.\\
For $A \subseteq \X$ and $\gamma \in \Gamma_{x,y}$ we define  
\begin{equation}
    \label{eq:posgamma_continuous_definition}
    \pos_{\gamma,A}(z):= \inf_{s \in \gamma^{-1}(z)} \ell(\gamma \cap A,s), \qquad \text{and} \qquad \pos_A(z):=\inf_{\gamma \in \Gamma_{x,y}} \pos_{\gamma,A}(z),
\end{equation}
where $\ell(\gamma \cap A,s) := \int_0^s \chi_A(t) \vert \dot{\gamma}(t)\vert \,\d t$ is the length spent by $\gamma$ inside $A$ up to time $s$. See Section \ref{sec:minkowski_vs_separatingratio} for a precise discussion. The steps sketched in the example above can be formally formulated as follows.

\begin{itemize}
    \item \textbf{Reduction to rectangles}. Let $A \subseteq \X$ be such that $\width_{x,y}(A) > 0$. Then $A' := A\cap\{ 0< \pos_A(z) < \width_{x,y}(A) \}$ satisfies $\SR_{x,y}(A') \le \SR_{x,y}(A).$ A proof of this fact can be obtained with the same methods of Proposition \ref{prop:level_set_position_function}.
    \item \textbf{Reduction to the Minkowski content}. Let $A \subseteq \X$ be such that $\width_{x,y}(A) > 0$. Then there exists $t\in (0,\width_{x,y}(A))$ such that $\Omega:= \{\pos_A \leq t \}$ is a topological separating set from $x$ to $y$ and $\mm^+(\Omega) \leq \Lambda\,\SR_{x,y}(A)$. This is the way we prove Theorem \ref{thm:main_equivalence_minima}.
\end{itemize}

We remark that $\partial \Omega \subseteq \{\pos_A = t \}$ plays the role of the line $L$ in the toy example. These and other procedures will be investigated in detail in Section \ref{sec:minkowski_vs_separatingratio}. \\
The key observation of the position function is that it allows to fibrate a separating set into separating sets by looking at its level sets.
Note indeed that this is the most delicate point: considering again the toy example of the Euclidean space, it is easy to find an example of a couple of points and separating set for which for many choice of Lipschitz functions their level sets are not in general separating sets.

Since we believe that the argument in Section \ref{sec:minkowski_vs_separatingratio} is technical, in order to isolate the main ideas behind the proof, we propose a similar argument in the case of measure graphs. This is the content of Section \ref{sec:graph_theory}. Since the objects are discrete and the argument is more combinatorial, we explain more in detail how the continuous argument reduces in this case. To be more precise, we define the discretized version of separating sets, size, separating ratio and Minkowski content. Finally, we provide an explicit procedure, that takes in input a discretized separating set and gives as output another discretized separating set with no larger separating ratio but with smallest possible discretized width. For such sets the computation of the discretized separating ratio coincides with the computation of the discretized Minkowski content. This provides the exact equality of the minima of the two discretized functionals, see Theorem \ref{thm:equivalence_minimum_energies_discrete}.\\

An open question is if the equality of infima on measure graphs in Theorem \ref{thm:equivalence_minimum_energies_discrete} could be used to prove equality of the infima in the continuous setting. 
Every metric space can be approximated by a suitable locally finite graph by looking at a $\delta$-net for $\delta\ll 1$ and associating to this set of points a structure of measure graph (see for instance the construction in \cite{GillLopez2015} and Remark \ref{rmk:graph_approximation}). %Then, we have the equivalence between minima of separating ratio and Minkowski content.
It would be interesting to understand the convergence of the discrete separating ratio and Minkowski content of the approximations to their continuous versions, but this is yet to be understood.
\subsection{Structure of the paper}
The paper is structured as follows. In Section \ref{sec:preliminaries} we recall definitions and properties we will need in the presentation. We recall preliminaries about curves, doubling measures, Poincar\'{e} inequality and Minkowski content. In Section \ref{sec:set-connectedness} we introduce the concept of $p$-set-connectedness and $p$-set-pencil and we prove some relations with the $p$-Poincar\'{e} inequality. We also show the complete equivalence for $p=1$ using Theorem \ref{theo:main-intro-p=1}. In Section \ref{sec:proofofthm1.4}, we prove Theorem \ref{theo:main-intro-p=1} and Corollary \ref{cor:Poincaré_p=1} following the scheme described in the introduction. In Section \ref{sec:minkowski_vs_separatingratio} we study the relation between the separating ratio and the Minkowski content of separating sets. The position function is introduced, its properties are studied and Theorem \ref{thm:main_equivalence_minima} is proved.
In Section \ref{sec:graph_theory} we introduce the discretized version of separating ratio, Minkowski content and position function and study the equivalent version of Theorem \ref{thm:main_equivalence_minima} in the setting of measure graphs. This section is self-contained and contains all the needed definitions for the discrete case. Finally, Appendix \ref{sec:revision_Poincare} contains a revised version of already known characterizations of Poincar\'{e} inequality via a curvewise approach, but formulated for a fixed couple of points. We also prove that all these conditions can be formulated equivalently either in terms of quasigeodesics with fixed bounded length or rectifiable curves.
\subsection{Acknowledgments}
The first author was supported by Academy of Finland, grants no. 321896. He currently acknowledges the support by the European Union’s Horizon 2020 research and innovation programme (Grant
agreement No. 948021). The first named author wishes to thank David Bate for fruitful discussions. We thank Tommaso Rossi and Pietro Wald for their comments on a preliminary version of the manuscript.

\section{Preliminaries}
\label{sec:preliminaries}
In this paper a metric measure space $(\X,\sfd,\mm)$ is a complete and separable metric space $(\X,\sfd)$ equipped with a nonnegative Borel measure $\mm$ that is finite on bounded sets. Given $A\subseteq \X$ and $r\geq 0$ we set $B_r(A):=\{ x:\, \sfd(x,A) < r\}$ and $\overline{B}_r(A):=\{ x:\, \sfd(x,A) \le r\}$, they are respectively the open and closed $r$-neighbourhood of $A$. For instance the open ball of center $x \in \X$ and radius $r$ is $B_r(x)$. We denote by $\mathscr{B}(\X)$ the class of Borel subsets of $\X$.\\
We denote by ${\rm Lip}(\X)$ the space of Lipschitz functions on $\X$ with values in $\R$. 
%For $u\in {\rm Lip(\X)}$ we denote by $\Lip(u)$ the Lipschitz constant of $u$. 
A function $u\colon \X \to \R$ is locally Lipschitz if for every $x\in \X$ there exist $r_x > 0$ and $\Lambda_x \ge 0$ such that $u\restr{B_{r_x}(x)}$ is $\Lambda_x$-Lipschitz. The set of maps that are locally Lipschitz is denoted by $\Lip_\loc(\X)$. Every locally Lipschitz map is Lipschitz on a compact set.
\begin{lemma}[{\cite[Theorem 4.2]{BeerGarrido2015}}]
\label{lemma:loclip_to_lip_up_to_scale}
    Let $(\X,\sfd)$ be a metric space, $K\subseteq \X$ be compact and $u \in \Lip_\loc(\X)$. Then $u\restr{K}$ is Lipschitz.
\end{lemma}

The \emph{local Lipschitz constant} $\lip u \colon \X \to \mathbb{R}$ of a function $u\colon \X \to \R$ is
$$\lip u (x):=\lims_{y \to x} \frac{|u(y)-u(x)|}{\sfd(y,x)} = \lim_{\delta \to 0} \sup_{y\in B_\delta(x)} \frac{|u(y)-u(x)|}{\sfd(y,x)} $$ 
with the convention that $\lip u (x)=0$ if $x$ is an isolated point. 

\subsection{Curves in metric spaces}

Let $C([0,1],\X)$ be the space of continuous curves $\gamma \colon [0,1]\to \X$ endowed with the sup distance $\sfd_\infty$, i.e.\ $\sfd_\infty(\gamma, \gamma'):=\sup_{t\in [0,1]}\sfd(\gamma(t),\gamma'(t))$ for $\gamma, \gamma' \in C([0,1],\X)$. Since $(\X,\sfd)$ is complete and separable, so $(C([0,1],\X),\sfd_{\infty})$ is.
The length of a curve is defined as
\begin{equation*}
    \ell(\gamma):= \sup \left\{ \sum_{i=0}^{N-1} \sfd(\gamma(t_i), \gamma(t_{i+1})), \,\{ t_i\}_{i=1}^N\subseteq [0,1],\,N \in \mathbb{N},\,
    0=t_0 < t_1 <\dots < t_N =1\right\}.
\end{equation*}
A curve $\gamma \in C([0,1],\X)$ is said to be rectifiable if $\ell(\gamma)<\infty$. Every rectifiable curve admits a Lipschitz reparametrization, i.e.\ there exists an increasing map $\phi \colon [0,1] \to [0,1]$ with $\phi(0)=0$ and $\phi(1)=1$ such that $\tilde{\gamma}:=\gamma \circ \phi$ is $\ell(\gamma)$-Lipschitz. \\
Given $x,y\in \X$ and $\Lambda \geq 1$ we define the set of $\Lambda$-quasigeodesics from $x$ to $y$ as
$$\Gamma_{x,y}^\Lambda := \{ \gamma \in C([0,1],\X)\,:\, \gamma(0) = x, \gamma(1) = y, \gamma \text{ is } \, \Lambda\sfd(x,y)\text{-Lipschitz} \}.$$
Observe that, by definition, every $\gamma \in \Gamma_{x,y}^\Lambda$ satisfies $\ell(\gamma) \leq \Lambda\sfd(x,y)$. Moreover every curve $\gamma$ satisfying $\ell(\gamma) \leq \Lambda\sfd(x,y)$ admits a reparametrization belonging to $\Gamma_{x,y}^\Lambda$. The set $\Gamma_{x,y}^\Lambda$ is compact if $\X$ is a proper metric space, as follows by Ascoli-Arzelà Theorem. We also define $\Gamma_{x,y} := \bigcup_{\Lambda \geq 1} \Gamma_{x,y}^\Lambda$. Notice that every rectifiable curve from $x$ to $y$ admits a reparametrization belonging to $\Gamma_{x,y}$. \\
The concatenation of two curves $\gamma,\eta \in C([0,1],\X)$ with $\eta(0) = \gamma(1)$ is the curve $\gamma \star \eta \in C([0,1],\X)$ defined as $(\gamma \star \eta)(t) = \gamma(2t)$ if $0\le t \le \frac{1}{2}$ and $(\gamma \star \eta)(t) = \eta(2t - 1)$ if $\frac{1}{2} \le t \le 1$. The concatenation of two Lipschitz curves is still Lipschitz. In particular if $x\neq y \neq z \neq x$ are points of $\X$ then $\gamma \star \eta \in \Gamma_{x,y}$ for every $\gamma \in \Gamma_{x,z}$ and every $\eta \in \Gamma_{z,y}$. Given a curve $\gamma \in C([0,1],\X)$ we define its reverse as $-\gamma \in C([0,1],\X)$, where $(-\gamma)(t) = \gamma(1-t)$. If $\gamma \in C([0,1],\X)$ and $0\le s_1 \le s_2 \le 1$ then the restriction of $\gamma$ to $[s_1,s_2]$ is the curve $\gamma\restr{[s_1,s_2]} \in C([0,1],\X)$ defined by $\gamma\restr{[s_1,s_2]}(t) = \gamma(s_1 +t(s_2-s_1))$ if $s_1 \le s_2$. If $\gamma$ is Lipschitz then also $\gamma\restr{[s_1,s_2]}$ is Lipschitz.
\vspace{2mm}

For every Lipschitz curve $\gamma\colon [0,1] \to \X$ the function $|\dot{\gamma}(\cdot)|\colon [0,1] \to [0,\infty)$ defined as $|\dot{\gamma}(t)|:=\lim_{h \to 0}\frac{\sfd(\gamma(t+h),\gamma(t))}{|h|} \in [0,\infty)$ is almost everywhere well defined and it is called the metric speed of $\gamma$. In particular this holds for all curves in $\Gamma_{x,y}$, for every $x,y\in \X$.
%
\begin{comment}
    We say that $\gamma \in C([0,1],\X)$ belongs to $AC([0,1],\X)$ if there exists $0 \le h \in L^1(\mm)$ such that $\sfd(\gamma_1,\gamma_0) \le \int_0^1 h(t) \,\d t$. If $\gamma \in AC([0,1],\X)$ then $\lim_{h \to 0} \frac{\sfd(\gamma_{t+h},\gamma_{t})}{|h|}=|\dot{\gamma}_t|$ exists for a.e.\ $t \in [0,1]$ and it is a finite number and $(t \mapsto |\dot{\gamma}_t|) \in L^1([0,1])$. The function $|\dot{\gamma}_t|$ is called the metric speed of $\gamma$. In particular, ${\rm Lip}([0,1],\X)\subseteq AC([0,1],\X)$ and $|\dot{\gamma}_t| \le {\rm Lip}(\gamma)$ for a.e.\ $t \in [0,1]$.
\end{comment}
%
Given a Borel function $g \colon \X \to [0,\infty)$ and $\gamma \colon [0,1]\to \X$ Lipschitz we define
\begin{equation*}
    \int_\gamma g\,\d s := \int_0^1 g(\gamma(t))\,|\dot{\gamma}(t)|\,\d t.
\end{equation*}
This quantity does not depend on the possible Lipschitz reparametrizations of $\gamma$. \\
Let $u\colon \X \to \R$ be a function. A function $g\colon \X \to [0,+\infty]$ such that
$$\vert u(\gamma(1)) - u(\gamma(0))\vert \leq \int_\gamma g\,\d s$$
for every rectifiable curve $\gamma$ is called an \emph{upper gradient} of $u$. The integral on the right is made along any Lipschitz reparametrization of $\gamma$. The set of upper gradients of $u$ is denoted by ${\rm UG}(u)$.

\subsection{Connectivity properties}

In the sequel we will use often assume that our metric space $(\X,\sfd)$ satisfies some connectivity properties that we now define in detail. The metric space $(\X,\sfd)$ is said to be
\begin{itemize}
    \item[-] path connected if for every $x,y\in \X$ there exists a continuous path $\gamma \colon [0,1] \to \X$ with $\gamma(0) = x, \gamma(1) = y$;
    \item[-] rectifiable path connected if $\Gamma_{x,y} \neq \emptyset$ for every $x,y\in \X$;
    \item[-] $\Lambda$-quasigeodesic if $\Gamma_{x,y}^\Lambda \neq \emptyset$ for every $x,y\in \X$;
    \item[-] pointwise rectifiable path connected if for every $x\in \X$ there exists $r_x > 0$ such that $\Gamma_{x,y} \neq \emptyset$ for every $y\in B_{r_x}(x)$;
    \item[-] pointwise quasigeodesic at $x\in \X$ if there exist $r_x > 0$ and $\Lambda_x \ge 1$ such that $\Gamma_{x,y}^{\Lambda_x} \neq \emptyset$ for every $y\in B_{r_x}(x)$;
    \item[-] pointwise quasigeodesic if it is pointwise quasigeodesic at every $x\in \X$;
    \item[-] locally $\Lambda$-quasigeodesic if for every $x\in \X$ there exist $r_x > 0$ such that $\Gamma_{y,z}^{\Lambda} \neq \emptyset$ for every $y,z\in B_{r_x}(x)$.
\end{itemize}
There are other possible variants but we will not need them. If $\Lambda = 1$ in the definitions involving quasigeodesics, then we say that the space is (pointwise or locally or globally) geodesic.\\
There are trivial relations among the conditions. For instance if $(\X,\sfd)$ is pointwise quasigeodesic then it is pointwise rectifiable path connected. Another easy one is the following.
\begin{lemma}
\label{lemma:rectifiable_path}
    Let $(\X,\sfd)$ be a path connected, pointwise rectifiable path connected metric space. Then it is rectifiable path connected.
\end{lemma}
\begin{proof}
    Let $x,y\in \X$ and $\gamma \colon [0,1] \to \X$ be a continuous path such that $\gamma(0)=x$ and $\gamma(1)=y$. For every $t\in [0,1]$ let $r_{\gamma(t)}>0$ be the number provided by the definition of pointwise rectifiable path connectedness. By compactness of $[0,1]$ we can find a finite set of points $\{t_0 = 0, t_1,\ldots,t_N = 1\}$ such that $\sfd(\gamma(t_i), \gamma(t_{i+1})) \leq r_{\gamma(t_i)}$. For every $i=0,\ldots,N-1$ let $\eta_i \in \Gamma_{\gamma(t_i), \gamma(t_{i+1})}$ and define $\eta = \eta_0 \star \eta_1 \star \cdots \star \eta_{N-1}$. Then $\eta \in \Gamma_{x,y}$.
\end{proof}

Observe that there are locally geodesic spaces that are no connected.

\subsection{Measure properties and Poincaré inequality}

\begin{comment}
We define some notions which are useful through the presentation. Let us fix two points $x,y \in \X$ and $L \ge 1$.
\begin{equation*}
\begin{aligned}
    &\Gamma_{x,y}:=\{ \gamma \in C([0,1],\X),\,\gamma_0=x,\,\gamma_1=y\}\\
    &\Gamma_{x,y}^{{\rm rect}}:=\{ \gamma \in C([0,1],\X),\,\gamma_0=x,\,\gamma_1=y,\,\ell(\gamma)<\infty \}. \\
    &\Gamma_{x,y}^L:=\{ \gamma \in C([0,1],\X),\,\gamma_0=x,\,\gamma_1=y,\,\ell(\gamma)\le L\,\sfd(x,y) \}.\\
\end{aligned}
\end{equation*}
We say that a metric space is $L$-quasiconvex if $\Gamma_{x,y}^L \neq \emptyset$ for every $x,y \in\X$. We call an element of $\Gamma_{x,y}^L$ a $L$-quasigeodesics joining $x$ to $y$. We say that $\gamma \in C([0,1],\X)$ is a geodesic if $\ell(\gamma)=\sfd(x,y)$, i.e.\ it is a $1$-quasigeodesics joining its endpoints.

A metric space $(\X,\sfd)$ is said to be locally rectifiable path connected if every $x\in \X$ has a neighbourhood $U$ such that $\Gamma_{x,z}^{{\rm rect}} \neq \emptyset$ for all $z\in U$. It is easy to see that the set $\lbrace z \in \X \text{ s.t. } \Gamma_{x,z}^{{\rm rect}} \neq \emptyset \rbrace$ is the maximal rectifiable path connected set containing $x$, i.e. it is the rectifiable path connected component of $x$. If $\X$ is locally rectifiable path connected then each rectifiable path connected component is open.    
\end{comment}

A metric measure space is said to be doubling if there exists $C_D >0$ such that for every $x \in \X$ and $r >0$ we have
\begin{equation*}
\mm(B(x,2r)) \le C_D \mm(B(x,r)).
\end{equation*}
We refer to $C_D$ as the doubling constant or we say that $(\X,\sfd,\mm)$ is $C_D$-doubling. A consequence of the definition of the doubling assumption is that $(\X,\sfd)$ is proper, i.e.\ closed and bounded sets are compact.\\
We recall the definition of the (local) Hardy Littlewood maximal function. Given $f \in L^1_{\rm loc}(\X)$ and $R >0$, we define
\begin{equation*}
   M_{R} f (x):=\sup_{0<r<R} \dashint_{B_r(x)} |f|\,\d \mm   
\end{equation*}
A metric measure space $(\X,\sfd,\mm)$ satisfies a (weak) $p$-Poincar\'{e} inequality if there exists $C_{P} \ge 1$ and $\lambda \ge 1$ such that 
\begin{equation}
\label{pPoincare-lip}
    \dashint_{B_r(x)} \left|u - \dashint_{B_r(x)}u\,\d \mm\right|\,\d \mm \le C_P r\,\left( \dashint_{B_{\lambda r}(x)} (\lip u)^p\,\d \mm \right)^{\frac{1}{p}}
\end{equation}
for every $u \in {\rm Lip}(\X)$. This is equivalent to the following condition: there exists $C_{P} \ge 1$ and $\lambda \ge 1$ such that 
\begin{equation}
\label{pPoincare-uppergrad}
    \dashint_{B_r(x)} \left|u - \dashint_{B_r(x)}u\,\d \mm\right|\,\d \mm \le C_P r\,\left( \dashint_{B_{\lambda r}(x)} g^p\,\d \mm \right)^{\frac{1}{p}}
\end{equation}
for every $u\colon \X \to \R$ Borel and every $g\in \UG(u)$. The equivalence between these two different conditions is proved in \cite[Theorem 2]{Kei03}.
\vspace{2mm}

The Poincaré inequality on a doubling metric measure space has several equivalent definitions that will be studied in big detail in the Appendix. Most of them are expressed in terms of the Riesz potential. Let $(\X,\sfd,\mm)$ be a metric measure space. Given $x,y\in \X$, the \emph{Riesz potential with poles at $x$ and $y$} $R_{x,y} \colon \X \to [0,\infty)$ is defined for every $z \in \X\setminus \lbrace x,y \rbrace$ as
\begin{equation}
\begin{aligned}
    R_{x,y}(z):&= \frac{\sfd(x,z)}{\mm(B_{\sfd(x,z)}(x))} + \frac{\sfd(y,z)}{\mm(B_{\sfd(y,z)}(y))}=: R_x(z)+R_y(z).
\end{aligned}
\end{equation}
Moreover, we define $R_{x,y}(x) = R_{x,y}(y) = 0$. For $L\geq 1$ we set $B_{x,y}^L := B_{2L\sfd(x,y)}(x)\cup B_{2L\sfd(x,y)}(y)$ and $\overline{B}_{x,y}^L := \overline{B}_{2L\sfd(x,y)}(x)\cup \overline{B}_{2L\sfd(x,y)}(y)$. The $L$-\emph{truncated Riesz potential with poles at $x,y$} is
\begin{equation}
    R_{x,y}^L(z):= \chi_{B_{x,y}^L}(z) R_{x,y}(z)
\end{equation}
for every $z \in \X\setminus \lbrace x,y \rbrace$. The corresponding Riesz measure is defined as
\begin{equation}
    \mm_{x,y}^L = R^L_{x,y}\,\mm. 
\end{equation}
It is a measure on $\X$ which is supported on $\overline{B}_{x,y}^L$. It has been studied for instance in \cite{Hei01}, \cite{Kei03} and \cite{CaputoCavallucci2024}. We recall one of its basic properties.
\begin{lemma}[\text{\cite[p. 72]{Hei01}, \cite[Prop. 2.2]{CaputoCavallucci2024}}]
\label{prop:properties_mxy}
    Let $(\X,\sfd,\mm)$ be a $C_D$-doubling metric measure space and fix $x,y\in \X$ and $L\geq 1$. Then $\mm_{x,y}^L(\X) \leq 8C_DL\sfd(x,y)$. In particular $\mm_{x,y}^L$ is a finite Borel measure.
\end{lemma}

We state the pointwise version of the Poincaré inequality for doubling metric measure spaces.
\begin{proposition}[{\cite[Theorem 9.5]{Hei01}}, {\cite[Theorem 2]{Kei03}}]
\label{prop:Poincaré_equivalence_Pointwise}
    Let $(\X,\sfd,\mm)$ be a doubling metric measure space. The following are quantitatively equivalent:
    \begin{itemize}
        \item[(i)] $(\X,\sfd,\mm)$ satisfies a $p$-Poincar\'{e} inequality;
        \item[(ii)] $\exists C > 0$, $L\geq 1$ such that for every $x,y \in \X$ and every $u\colon \X \to \R$ Borel it holds
        \begin{equation}
        \label{eq:Riesz_PtPI}
            |u(x)-u(y)|^p\le C\, \sfd(x,y)^{p-1}\, \int_\X g^p\,\d \mm_{x,y}^{L}
        \end{equation} 
        for every $g\in \UG(u)$.
        \item[(iii)] $\exists C > 0$, $L\geq 1$ such that for every $x,y \in \X$ and every $u\colon \X \to \R$ Lipschitz it holds
        \begin{equation}
        \label{eq:Riesz_PtPI_Lipu}
            |u(x)-u(y)|^p\le C\, \sfd(x,y)^{p-1}\, \int_\X \lip u^p\,\d \mm_{x,y}^{L}.
        \end{equation} 
    \end{itemize}
\end{proposition}

The pointwise estimate \eqref{eq:Riesz_PtPI} remains quantitatively equivalent to \eqref{eq:Riesz_PtPI_Lipu} even if we require it only for a fixed couple of points. But for that we need to assume some connectivity property of the space, which are ensured for instance when a global Poincaré inequality holds. The proof will be given in Appendix \ref{sec:revision_Poincare}.

\begin{lemma}
\label{lemma:pointwise_differenti}
    Let $(\X,\sfd,\mm)$ be a locally $\Lambda$-quasiconvex metric measure space and let $x,y\in\X$. Then the following conditions are quantitatively equivalent:
    \begin{itemize}
        \item[(i)] there exist $C>0,L\ge 1$ such that $|u(x)-u(y)|^p\le C \sfd(x,y)^{p-1} \int_\X g^p(z)\,\d \mm_{x,y}^{L}(z)$ for every $u$ Borel and every $g\in \UG(u)$;
        \item[(ii)] there exist $C>0,L\ge 1$ such that $|u(x)-u(y)|^p\le C \sfd(x,y)^{p-1} \int_\X \lip u(z)^p\,\d \mm_{x,y}^{L}(z)$ for every $u \in \Lip(\X)$.
    \end{itemize}
\end{lemma}

\subsection{Minkowski content}
\label{subsec-Minkowski}

We recall the definition of Minkowski content of a set.
\begin{definition}
    Let $(\X,\sfd,\mm)$ be a metric measure space and let $A\subseteq \X$ be Borel. The Minkowski content of $A \subseteq \X$ is
\begin{equation*}
    \mm^{+}(A):=\limi_{r \to 0} \frac{\mm\left(\overline{B}_r(A) \setminus A \right)}{r}.
\end{equation*}
\end{definition}
%
%If $p=1$, we use the shorthand notation $\mm^{-}(A)$ for $\mm^{-,1}(A)$ and $\mm^{+}(A)$ for $\mm^{+,1}(A)$.

The Minkowski content satisfies a coarea type inequality.
\begin{proposition}[\text{%\cite[Prop. \ 3.5]{KorteLahti2014}, 
\cite[Lemma 3.2]{AmbDiMarGig17}}]
\label{prop:coarea_inequality_minkowski}
    Let $(\X,\sfd,\mm)$ be a metric measure space and suppose $\mm$ is finite. Then for every bounded $u \in {\rm Lip}(\X)$ we have
    \begin{equation}
        \label{eq:coareain3_AGD}
        \int_{-\infty}^{\infty} \mm^{+}(\{ u \geq t\}) \,\d t \le \int_\X {\rm \lip} u \,\d \mm.
    \end{equation}
    %and
    %\begin{equation}
    %    \label{eq:coareain3_KL}
    %    \int_{-\infty}^{\infty} \mm^{+}(\partial \{ u > t\}) \,\d t \le 2\int_\X {\rm \lip} u \,\d \mm.
    %\end{equation}
\end{proposition}

\section{The $p$-set-connectedness and the $p$-set-pencil}
\label{sec:set-connectedness}

The goal of this section is to introduce new natural conditions on a metric measure space and show their relation with the Poincar\'{e} inequality. These conditions are inspired by the one studied in \cite{ErikssonBique2019} and \cite{Sylvester-Gong-21}. Let $(\X,\sfd,\mm)$ be a metric measure space. Given a set $A\subseteq \X$ and a rectifiable curve $\gamma \colon [0,1] \to \X$ we define the length of $\gamma$ in $A$ as
$$\ell(\gamma \cap A) := \int_\gamma \chi_A \,\d s.$$
Let $x,y\in \X$, $x\neq y$ and $A\subseteq \X$ Borel. We define the $\Lambda$-width and the width of $A$
as
$$\width^\Lambda_{x,y}(A) := \inf_{\gamma \in \Gamma_{x,y}^\Lambda} \ell(\gamma \cap A), \text{ and respectively } \width_{x,y}(A) := \inf_{\gamma \in \Gamma_{x,y}} \ell(\gamma \cap A),$$
where $\Lambda\geq 1$. We use the usual convention $\inf \emptyset = +\infty$. In order to simplify the notations we will often consider $\width_{x,y}$ as the limit case of $\width_{x,y}^\Lambda$ with $\Lambda = \infty$. The same convention is used in the following definition.

\begin{definition}
    Let $C>0$, $L\geq 1$ and $\Lambda \in [1,+\infty]$. We say that $(\X,\sfd,\mm)$ is $(C,L,\Lambda)$ $p$-set-connected at $x,y\in \X$ if
\begin{equation}
    \label{eq:defin_set_connectedness_bounded}
    \left(\width_{x,y}^\Lambda(A)\right)^p\leq C\sfd(x,y)^{p-1}\mm_{x,y}^L(A) \text{ for all } A\subseteq \X  \text{ Borel.}
\end{equation}
If $\Lambda = +\infty$ we will simply say that $(\X,\sfd,\mm)$ is $(C,L)$ $p$-set-connected at $x,y\in \X$.\\
We say that $(\X,\sfd,\mm)$ has a $(C,L,\Lambda)$ $p$-set-pencil between $x,y\in \X$ if there exists $\alpha \in \mathscr{P}(\Gamma_{x,y}^\Lambda)$, i.e. $\alpha$ is a probability measure on the space $\Gamma_{x,y}^\Lambda$, such that
\begin{equation}
    \label{eq:defin_set_pencil_bounded}
    \left(\int \ell(\gamma \cap A) \,\d\alpha(\gamma)\right)^p\leq C\sfd(x,y)^{p-1}\mm_{x,y}^L(A) \text{ for all } A\subseteq \X  \text{ Borel}.
\end{equation}
Also in this case if $\Lambda = +\infty$ we will simply say that $(\X,\sfd,\mm)$ has a $(C,L)$ $p$-set-pencil between $x,y\in \X$.
\end{definition}

These definitions are special cases of properties that are equivalent to the pointwise Poincaré inequality as studied in Appendix \ref{sec:revision_Poincare}. For instance the $(C,L,\Lambda)$ $p$-set-connectedness property correspond to condition (ii) of Theorem \ref{thm:equivalence_pencil_Ap-connectedness} applied to the functions $\chi_A$, when $\Lambda < +\infty$. When $\Lambda = +\infty$ it is related to condition (iii) of Theorem \ref{thm:equivalence_pencil_Ap-connectedness}, which is formulated in terms of $1$-Lipschitz functions. A similar relation holds between the $p$-set-pencil conditions and conditions (iv) and (v) of Theorem \ref{thm:equivalence_pencil_Ap-connectedness}.\\
There are easy relations between the notions introduced above and the validity of the pointwise estimate \eqref{eq:Riesz_PtPI}. 

\begin{lemma}
\label{lemma:Poincaré_to_set_connectedness}
    Let $(\X,\sfd,\mm)$ be a path connected, pointwise quasiconvex doubling metric measure space. Let $x,y\in \X$. Then
    \begin{itemize}
        \item[(i)] if $(\X,\sfd,\mm)$ has a $(C,L,\Lambda)$ $p$-set-pencil between $x,y$ then it has a $(C,L)$ $p$-set-pencil between $x,y$;
        \item[(ii)] if $(\X,\sfd,\mm)$ has a $(C,L,\Lambda)$ $p$-set-pencil between $x,y$ then it is $(C,L,\Lambda)$ $p$-set-connected at $x,y$;
        \item[(iii)] if $(\X,\sfd,\mm)$ is $(C,L,\Lambda)$ $p$-set-connected at $x,y$ then it is $(C,L)$ $p$-set-connected at $x,y$;
        \item[(iv)] if \eqref{eq:Riesz_PtPI} holds at $x,y\in \X$ then $(\X,\sfd,\mm)$ has a $(C,L,\Lambda)$ $p$-set-pencil between $x,y$, quantitatively.
    \end{itemize}
\end{lemma}

\begin{proof}
    The statements (i), (ii) and (iii) are trivial, so we just focus on (iv).
    We assume that \eqref{eq:Riesz_PtPI} holds at $x,y$. By item (iv) in Theorem \ref{thm:equivalence_pencil_Ap-connectedness}, we can find $C>0,L\ge 1$ and $\alpha \in \mathscr{P}(\Gamma_{x,y}^L)$ such that  
    $$\left(\int\int_\gamma \chi_A\,\d s\,\d \alpha(\gamma)\right)^p \leq C\sfd(x,y)^{p-1}\int \chi_A^p\,\d\mm_{x,y}^L=C\sfd(x,y)^{p-1}\,\mm_{x,y}^L(A),$$
    for every $A\subseteq \X$ Borel. This gives the $(C,L,L)$ $p$-set-pencil between $x$ and $y$. 
\end{proof}
As a corollary, we have the following statement.
\begin{corollary}
    Let $(\X,\sfd,\mm)$ be a doubling metric measure space. If $(\X,\sfd,\mm)$ satisfies a $p$-Poincaré inequality then there exists $C > 0, L,\Lambda \ge 1$ such that $(\X,\sfd,\mm)$ has a $(C,L,\Lambda)$ $p$-set-pencil between $x,y$ and it is $(C,L,\Lambda)$ $p$-set-connected at $x,y$, for every $x,y \in \X$.
\end{corollary}
\begin{proof}
    Since $(\X,\sfd,\mm)$ satisfies a $p$-Poincaré inequality then \eqref{eq:Riesz_PtPI} holds for every couple of points $x,y\in \X$ with constants depending only on the constants of the Poincaré inequality and the doubling constant, as recalled in Proposition \ref{prop:Poincaré_equivalence_Pointwise}. The thesis follows by Lemma \ref{lemma:Poincaré_to_set_connectedness}.
\end{proof}

%\begin{definition}
%    We say that $(\X,\sfd,\mm)$ is $(C,L)$-boundedly (resp. $(C,L)$-unboundedly) $p$-set-connected if it is $(C,L)$-boundedly (resp. unboundedly) $p$-set-connected at $x,y \in \X$ for all $x,y\in \X$. \\
%    We say that $(\X,\sfd,\mm)$ has a $(C,L)$-bounded (resp. unbounded) $p$-set-pencils if it has a $(C,L)$-bounded (resp. unbounded) $p$-set-pencil between $x,y \in \X$ for all $x,y\in \X$.
%\end{definition}

There is a partial converse of this corollary, essentially due to \cite{Sylvester-Gong-21}.

\begin{proposition}
\label{prop:equivalence_connectivity_Sylvester}
Let $(\X,\sfd,\mm)$ be a doubling metric measure space. Suppose that there exist $C > 0, L \geq 1$ such that it is $(C,L,L)$ $p$-set-connected at $x,y$ for every $x,y \in \X$. Then it satisfies a $q$-Poincaré inequality for all $q>p$.
\end{proposition}

Observe that the assumption $L = \Lambda$ in the statement is not restrictive.

\begin{proof}
    We claim that $(\X,\sfd,\mm)$ is $(L',C',p)$-max connected in the sense of \cite[Definition 2.12]{Sylvester-Gong-21}, where $L'=2L+1$ and $C' = (7CC_DL)^{\frac{1}{p}}$. \\
    Let us take $A \subseteq \X$ with $\max\lbrace M_{(2L+1)\sfd(x,y)}(\chi_A)(x), M_{(2L+1)\sfd(x,y)}(\chi_A)(y) \rbrace < \tau$, $\tau \in (0,1]$. For every $\varepsilon > 0$ we can find, by \eqref{eq:defin_set_connectedness_bounded}, a curve $\gamma_\varepsilon \in \Gamma_{x,y}^L$ such that $\ell(\gamma_\varepsilon \cap A) \leq C^{\frac{1}{p}}\sfd(x,y)^{\frac{p-1}{p}}\mm_{x,y}^L(A)^\frac{1}{p} + \varepsilon$.    
    We recall the classical bound (see \cite[p.72]{Hei01}): for every Borel $0\le g$ and $\rho >0$, we have $\int_{B_{\rho}(x)} g\,R_x\,\d \mm \le \rho M_{\rho}g(x)$ for every $x \in \X$. Applying it, we can estimate
    \begin{equation}
        \begin{aligned}
            \mm_{x,y}^L(A) &= \int_{B_{x,y}^L} \chi_A(z)(R_x(z) + R_y(z))\, \d\mm \\ &\leq \int_{B_{(2L+1)\sfd(x,y)}(x)} \chi_A(z)R_x(z)\,\d\mm + \int_{B_{(2L+1)\sfd(x,y)}(y)} \chi_A(z)R_y(z)\,\d\mm \\
            &\leq C_D(2L+1)\sfd(x,y)\left(M_{(2L+1)\sfd(x,y)}(\chi_A)(x) + M_{(2L+1)\sfd(x,y)}(\chi_A)(y)\right) \\
            &\leq 2C_D(2L+1)\sfd(x,y)\tau.
        \end{aligned}
    \end{equation}
    So we get $\ell(\gamma_\varepsilon \cap A) \leq (6CC_DL)^{\frac{1}{p}}\tau^\frac{1}{p}\sfd(x,y) + \varepsilon$, where we used that $L\geq 1$ so $2L+1 \leq 3L$. Choosing $\varepsilon$ small enough we get $\ell(\gamma_\varepsilon \cap A) \leq (7CC_DL)^{\frac{1}{p}}\tau^\frac{1}{p}\sfd(x,y)$. Set $\gamma := \gamma_\varepsilon$ for this particular choice of $\varepsilon$. Then we have $\gamma \in \Gamma_{x,y}^L \subseteq \Gamma_{x,y}^{L'}$ and $\ell(\gamma \cap A)\leq C'\tau^\frac{1}{p}\sfd(x,y)$. This shows that $(\X,\sfd,\mm)$ is $(L',C',p)$-max connected.\\
    The fact that $(\X,\sfd,\mm)$ satisfies a $q$-Poincaré inequality for all $q>p$ follows by \cite[Theorem 2.19 \& Lemma 2.20]{Sylvester-Gong-21}.
\end{proof}

\begin{remark}
    For the proof of the $q$-Poincaré inequality using the methods of \cite{Sylvester-Gong-21} two things are needed: a bound to the length of the curves involved in the max-connectedness property of \cite{Sylvester-Gong-21}, which corresponds to consider the version of our $p$-set-connectedness with $\Lambda < \infty$, and the validity of \eqref{eq:defin_set_connectedness_bounded} \emph{for all couple} of points $x,y\in \X$. Both these properties are in contrast with the equivalences we show in Appendix \ref{sec:revision_Poincare}. It is unclear to us if these two properties are necessary to infer the validity of \eqref{eq:Riesz_PtPI} for $q>p$ at fixed points $x,y\in \X$. There are two difficulties in adapting the proof of Theorem \ref{thm:equivalence_pencil_Ap-connectedness} in this setting. First of all it is not clear how to get a $p$-set-pencil from a $p$-set-connectedness. Secondly it is not clear if the existence of a $p$-set-pencil \emph{between $x,y$} implies \eqref{eq:Riesz_PtPI} for $q>p$ at $x,y$. 
\end{remark}

In the special case $p=1$ we can answer completely in the affirmative to the questions raised in the previous remark, and the situation is completely similar to what we see in Appendix \ref{sec:revision_Poincare}. In particular in the next theorem, which is the setwise analogous of Theorem \ref{thm:equivalence_pencil_Ap-connectedness}, we see how the equivalences work if we only know the conditions at two points $x,y$. This is completely new since, as recalled above, the methods of \cite{Sylvester-Gong-21} do not provide a $1$-Poincaré inequality under a $(L,C,1)$-max connectedness assumption, but only a $q$-Poincaré inequality for every $q>1$. The other difference, as already remarked, is that we do not need that the assumption holds for every couple of points, but just for a fixed one.

\begin{theorem}
\label{thm:equivalence_1-connectivity}
    Let $(\X,\sfd,\mm)$ be a path connected, locally $\Lambda$-quasiconvex, doubling metric measure space and let $x,y\in \X$. The following conditions are quantitatively equivalent:
    \begin{itemize}
        \item[(i)] $(\X,\sfd,\mm)$ satisfies \eqref{eq:Riesz_PtPI} at $x,y$ for $p=1$;
        \item[(ii)] $(\X,\sfd,\mm)$ is $(C,L,\Lambda')$ $1$-set-connected at $x,y$;
        \item[(iii)] $(\X,\sfd,\mm)$ is $(C,L)$ $1$-set-connected at $x,y$;
        \item[(iv)] $(\X,\sfd,\mm)$ has a $(C,L,\Lambda')$ $1$-set-pencil between $x,y$;
        \item[(v)] $(\X,\sfd,\mm)$ has a $(C,L)$ $1$-set-pencil between $x,y$.
    \end{itemize}
\end{theorem}
The proof is a consequence of Theorem \ref{theo:main-intro-p=1}.
\begin{proof}
    By Lemma \ref{lemma:Poincaré_to_set_connectedness} we know that (i) implies (iv), (iv) implies (v) and (ii), (v) implies (ii) and (v) implies (iii). So it is enough to show that (iii) implies (i). For this it is enough to use Theorem \ref{theo:main-intro-p=1}.
\end{proof}

\section{The proof of Theorem \ref{theo:main-intro-p=1}}
\label{sec:proofofthm1.4}
In this section we prove Theorem \ref{theo:main-intro-p=1}. 
The $1$-set connectedness can be rephrased in the following way. For every subset $A\subseteq \X$ we define its $\Lambda$-separating ratio between $x,y$ as
$$\SR_{x,y}^\Lambda(A) := \frac{\mm_{x,y}^L(A)}{\width^\Lambda_{x,y}(A)}.$$
We also use the conventions that $\SR^\Lambda_{x,y}(A) = 0$ if $\width^\Lambda_{x,y}(A) = +\infty$, that happens if and only if $\Gamma^\Lambda_{x,y} = \emptyset$, and that $\SR^\Lambda_{x,y}(A) = +\infty$ if $\width_{x,y}^\Lambda(A) = 0$. As usual when $\Lambda = \infty$ we simply omit it from the notation.\\
So $(\X,\sfd,\mm)$ is $(C,L,\Lambda)$ $1$-set-connected if and only if 
\begin{equation}
\label{eq:SR_lowerbound}
    \inf_{A \subseteq \X} \SR_{x,y}^\Lambda(A) \geq \frac{1}{C}.
\end{equation}
On the other hand in \cite[Thm.\ 1.2]{CaputoCavallucci2024} the authors showed that the $1$-Poincaré inequality is equivalent to a lower bound on the Minkowski content of all separating sets. For the definition of separating sets we refer to Definition \ref{def:separating_sets}. More precisely it is there proved that a $1$-Poincaré inequality holds if and only if there exists $c>0$ and $L\geq 1$ such that for every $x,y \in \X$ one has
\begin{equation}
\label{eq:MC_lowerbound}
    \inf_{\Omega \in \SS_{\text{top}}(x,y)} (\mm_{x,y}^L)^+(\Omega) \geq c.
\end{equation}
The same statement is true if the condition hold only at two fixed points $x,y$. This is the content of Theorem \ref{theo:main-intro-p=1} that we reproduce here for the reader's convenience.
\begin{reptheorem}{theo:main-intro-p=1}
%\label{thm:equivalence_minima_to_pointwise_estimate}
    Let $(\X,\sfd,\mm)$ be a path connected, locally $\Lambda$-quasiconvex, doubling metric measure space and let $x,y\in \X$. Then the following conditions are quantitatively equivalent:
    \begin{itemize}
        \item[(i)] there exist $C >0$, $L \ge 1$ such that
        \begin{equation}
            |u(x)-u(y)| \le C \int \lip u \,\d \mm_{x,y}^L\qquad \text{for all }u \in {\Lip}(\X);
        \end{equation}
        \item[(ii)] there exist $C >0$, $L \ge 1$ such that
        \begin{equation}
        \label{eq:pointwise_Poincaré_Riesz_UG}
            |u(x)-u(y)| \le C \int g \,\d \mm_{x,y}^L\qquad \text{for all }u \text{ Borel and all } g \in \UG(u);
        \end{equation}
        \item[(iii)] the space $(\X,\sfd,\mm)$ is $(C,L)$ $1$-set connected at $x,y$, i.e. $\inf_{A\subseteq \X} \SR_{x,y}(A) \ge C^{-1}$;  
        \item[(iv)] there exist $c>0$, $L\ge 1$ such that $\inf_{A\subseteq \X,\, A \textup{ closed}} \SR_{x,y}(A) \ge C^{-1}$;
        \item[(v)] the space $(\X,\sfd,\mm)$ satisfies \textup{(BMC)}$_{x,y}$, i.e. $\inf_{\Omega \in \SS_{\textup{top}}(x,y)}(\mm_{x,y}^L)^+(\Omega) \ge c$ for some $L\ge 1$ and $c>0$.
    \end{itemize}
\end{reptheorem}

\begin{proof}
    The equivalence between (i) and (ii) is proved in Theorem \ref{thm:equivalence_pencil_Ap-connectedness}, see also Lemma \ref{lemma:pointwise_differenti}.
    By Lemma \ref{lemma:Poincaré_to_set_connectedness} we know that (ii) implies (iii). It is also clear that (iii) implies (iv). We assume (iv) and we take a separating set $\Omega \in \SS_{\textup{top}}(x,y)$. If $(\mm_{x,y}^L)^+(\Omega) = +\infty$ there is nothing to prove, so we suppose $(\mm_{x,y}^L)^+(\Omega) < +\infty$, which implies $\mm_{x,y}^L(\partial \Omega) = 0$. Consider the set $A_r:= \overline{B}_r(\Omega) \setminus \Int(\Omega)$, which is a closed subset of $\X$. Observe that $\width_{x,y}(A_r) \geq r$ if $r < \min\{\sfd(\partial \Omega, x), \sfd(\partial \Omega, y) \}$. Indeed every curve $\gamma \in \Gamma_{x,y}$ has to go from the inside of $\Omega$ to the outside of $\overline{B}_r(\Omega)$, because of our choice of $r$. Therefore its length inside $A_r$ is at least $r$. We can now compute
     $$(\mm_{x,y}^L)^+(\Omega) = \limi_{r \to 0}\frac{\mm_{x,y}^L(\overline{B}_r(\Omega) \setminus \Omega)}{r} = \limi_{r \to 0}\frac{\mm_{x,y}^L(A_r)}{r} \geq \limi_{r \to 0}\frac{\mm_{x,y}^L(A_r)}{\width_{x,y}(A_r)} \geq \inf_{\substack{A \subseteq \X \\ A \text{ closed}}} \SR_{x,y}(A) \geq c.$$
     The second equality follows by the fact that $\mm_{x,y}^L(\partial \Omega) = 0$, while the last inequality follows by our assumption (iv).
     Since this is true for every $\Omega \in \SS_{\textup{top}}(x,y)$ we get (v).\\
     The proof of (v) implies (i) is exactly the proof of the last implication in \cite[Theorem 6.1]{CaputoCavallucci2024}. For reader's convenience we report it here. Let $u\in \Lip(\X)$ and let $x,y\in \X$. We can assume that $u(x) < u(y)$ otherwise there is nothing to prove. The sets $\Omega_t := \lbrace u \geq t \rbrace$ belong to $\SS_\text{top}(x,y)$ for all $t\in (u(x),u(y))$. We can apply the coarea inequality \eqref{eq:coareain3_AGD} with respect to the measure $\mm_{x,y}^L$ to get
     $$c\, \vert u(x) - u(y) \vert \leq \int_{u(x)}^{u(y)} (\mm_{x,y}^L)^{+}(\lbrace u \geq t \rbrace) \,\d t \leq \int_\X \lip u \,\d\mm_{x,y}^L.$$
     Therefore (i) follows with $C = 1/c$.
\end{proof}

\begin{remark}
\label{rmk:no_connectivity}
We stress that the assumption that the metric space $(\X,\sfd)$ is path connected and locally $\Lambda$-quasiconvex only enters in the proof of (i) implies (ii) and (ii) implies (iii).
\end{remark}

\begin{proof}[Proof of Corollary \ref{cor:Poincaré_p=1}]
    If $(\X,\sfd,\mm)$ satisfies a $1$-Poincaré inequality then $(\X,\sfd)$ is $\Lambda$-quasiconvex (\cite[Proposition 8.3.2]{HKST15}) and \eqref{eq:Riesz_PtPI} holds. So  $(\X,\sfd,\mm)$ is $(C,L)$ $1$-set-connected at every couple of points by Theorem \ref{theo:main-intro-p=1}. Viceversa by the proof of Theorem \ref{theo:main-intro-p=1} and Remark \ref{rmk:no_connectivity} we conclude that if $(\X,\sfd,\mm)$ is $(C,L)$ $1$-set-connected at every couple of points then it satisfies \eqref{eq:Riesz_PtPI_Lipu} at every couple of points. By Proposition \ref{prop:Poincaré_equivalence_Pointwise} this implies that $(\X,\sfd,\mm)$ satisfies a $1$-Poincaré inequality.
\end{proof}

\section{Relation between Minkowski content and separating ratio}
\label{sec:minkowski_vs_separatingratio}
In the previous section, in particular in Theorem \ref{theo:main-intro-p=1}, we proved the equivalence between the pointwise estimate \eqref{eq:Riesz_PtPI} at two fixed points $x,y$ and a positive lower bound on either the separating ratio as expressed in \eqref{eq:SR_lowerbound} or the Minkowski content of topological separating sets as expressed in \eqref{eq:MC_lowerbound}. In particular a key step in the proof of Theorem \ref{theo:main-intro-p=1} was to show that the following inequality is always true:
\begin{equation}
\label{diseq:infima}
    \inf_{\Omega \in \SS_{\textup{top}}(x,y)} (\mm_{x,y}^L)^+(\Omega) \geq \inf_{A \subseteq \X\, \text{closed}} \SR_{x,y}(A).
\end{equation}
The goal of this section is to study in detail these two infima.
In this section we do not need assumptions on the measure $\mm$, except that it has compact support. So we will deal with the following general situation: $(\X,\sfd,\mm)$ is a metric measure space and, for given $x,y\in \X$ and $A\subseteq \X$, we set
$$\SR_{x,y}(A) := \frac{\mm(A)}{\width_{x,y}(A)}$$
with the same conventions as above, i.e. $\SR_{x,y}(A) = 0$ if $\width_{x,y}(A) = +\infty$ and $\SR_{x,y}(A) = +\infty$ if $\width_{x,y}(A) = 0$.
The relevant application is when the measure is $\mm_{x,y}^L$ for some doubling measure $\mm$ and some $L\ge 1$.

\subsection{Path separating sets}
 In order to proceed in this direction it is useful to define another class of separating sets that we call path-separating sets. It is the class
$$\SS_{\textup{path}}(x,y) := \{ A \subseteq \X \,:\, A \cap \gamma \neq \emptyset \text{ for all } \gamma \in \Gamma_{x,y} \}.$$
For completeness we discuss the relation between topological separating sets and closed path-separating sets. This is not strictly necessary for the proof of the equivalence between the minima of the two functionals in Section \ref{subsec:equivalence_minima}. 
\begin{lemma}
\label{lemma:SS_top<path}
    Let $(\X,\sfd)$ be a metric space and let $x,y\in \X$. If $\Omega \in \SS_{\textup{top}}(x,y)$ then $\partial \Omega \in \SS_{\textup{path}}(x,y)$.
\end{lemma}
\begin{proof}
    Let $\Omega \in \SS_{\text{top}}(x,y)$ and let $\gamma \in \Gamma_{x,y}$. The sets $\Int(\Omega)$ and $\Omega^c$ are open, so are $[0,1]\cap \gamma^{-1}(\Int(\Omega))$ and $[0,1]\cap \gamma^{-1}(\Omega^c)$. They cannot cover the whole $[0,1]$ by connectedness. So there exists $t\in [0,1]$ such that $\gamma(t) \in \partial \Omega$. This shows that $\partial \Omega \cap \gamma \neq \emptyset$ for all $\gamma \in \Gamma_{x,y}$, i.e. $\partial \Omega \in \SS_{\text{path}}(x,y)$.
\end{proof}

%To every path separating set we can associate an interior and an exterior set, together with a signed distance function. Let $x,y\in \X$ and $A\in \SS_{\text{path}}(x,y)$. We define the sets
%$$A_- := \{ z\in \X\,:\, \exists \gamma \in \Gamma_{x,z} \text{ such that } \gamma \cap A = \emptyset \}$$
%and
%$$A_+ := \{ z\in \X\,:\, \gamma \cap A \neq \emptyset \text{ for all } \gamma \in \Gamma_{x,z} \}.$$
%Observe that by definition $x\in A_-$ and $y \in A_+$ as soon as $\min\{ \sfd(x,A), \sfd(y,A)\} > 0$. The signed distance function from $A$ is
%$$\bar{\sfd}_A\colon \X \to \R, \quad z \mapsto \begin{cases}
%    \sfd(z,A) \text{ if } z \in A_+,\\
%    -\sfd(z,A) \text{ if } z \in A_-.
%\end{cases}$$

The next is a partial converse of this statement.
\begin{proposition}
\label{prop:SS_path<top}
    Let $(\X,\sfd)$ be a pointwise quasiconvex metric space and let $x,y\in \X$. If $A \in \SS_{\textup{path}}(x,y)$ is closed and satisfies $\min\{\sfd(x,A),\sfd(y,A)\}>0$ then there exists $\Omega \in \SS_{\textup{top}}(x,y)$ such that $\partial \Omega \subseteq A$.
\end{proposition}
\begin{proof}
    We claim that the set $\Omega$ defined as the closure of
    $$\tilde{\Omega} := \{ z\in \X\,:\, \exists \gamma \in \Gamma_{x,z} \text{ such that } \gamma \cap A = \emptyset \}$$
    satisfies the thesis. First of all $\tilde{\Omega}$ is open. Indeed let $z\in \tilde{\Omega}$, let $\gamma \in \Gamma_{x,z}$ such that $\gamma \cap A = \emptyset$ and let $r = \sfd(z,A) > 0$ since $A$ is closed. Let $\Lambda_z$ and $r_z$ be the constants of the pointwise quasiconvexity at $z$. We claim that $B_{r'}(z) \subseteq \tilde{\Omega}$ if $r' < \min\{r_z, \frac{r}{\Lambda_z} \}$. For this take $w\in B_{r'}(z)$ and a curve $\gamma_w \in \Gamma_{z,w}^{\Lambda_z}$. Observe that $\gamma_w\cap A = \emptyset$ since $\sfd(\gamma_w(t), z) \le \ell(\gamma_w) \leq \Lambda_z r' < r$, so $\gamma_w(t) \notin A$ for every $t$.    
    The concatenation $\gamma \star \gamma_w$ belongs to $\Gamma_{x,w}$ and does not intersect $A$, i.e. $w\in \tilde{\Omega}$.\\
    Secondly $x\in \tilde{\Omega}$ because $\sfd(x,A) > 0$, so $\tilde{\Omega}$ contains a small ball around $x$ since it is open. \\
    Hence the claim follows if we prove that $\partial \Omega=\partial \tilde{\Omega} \subseteq A$. Indeed if this is the case then $y\in \tilde{\Omega}^c$ and $\sfd(y,\partial \Omega) \geq \sfd(y,A) > 0$, showing that there is a small ball around $y$ contained in $\Omega^c$.\\
\begin{comment}
    We show now that $\tilde{\Omega}^c$ is closed. Let $z_n$ be a sequence of points in $\tilde{\Omega}^c$ converging to $z$. By definition, every curve in $\Gamma_{x,z_n}$ has to intersect $A$. Suppose that $z \in \tilde{\Omega}$, and so that there exists a curve $\gamma \in \Gamma_{x,z}$ such that $\gamma \cap A = \emptyset$. Since $(\X,\sfd)$ is pointwise quasigeodesic we can find curves $\eta_n \in \Gamma_{z,z_n}^{\Lambda_z}$ for some fixed $\Lambda_z > 0$. Let $\gamma_n$ be the concatenation $\gamma \star \eta_n$. It belongs to $\Gamma_{x,z_n}$ so it must intersect $A$. Since $\gamma \cap A = \emptyset$ then there must be a point $w_n$ along $\eta_n$ that belongs to $A$. But $\sfd(z,w_n) \leq \ell(\eta_n) \leq \Lambda_z\sfd(z,z_n) \to 0$ as $n$ goes to infinity. Since $A$ is closed we deduce that $z\in A$, implying that $\gamma \cap A \neq \emptyset$, a contradiction.\\
\end{comment}
    It remains to prove that $\partial \tilde{\Omega} \subseteq A$. Let $z$ be a point of $\partial \tilde{\Omega}$. By definition we can find points $z_n \in \tilde{\Omega}$ and $w_n \in \tilde{\Omega}^c$ converging to $z$. Let $r_z,\Lambda_z$ be the constants of the pointwise quasiconvexity at $z$. We can suppose that $n$ is big enough to have $z_n,w_n \in B_{r_z}(z)$. Hence for every $n$ we fix a curve $\gamma_n \in \Gamma_{x,z_n}$ such that $\gamma_n \cap A = \emptyset$, a curve $\eta_n \in \Gamma_{z_n,z}^{\Lambda_z}$ and a curve $\xi_n \in \Gamma_{z,w_n}^{\Lambda_z}$. Finally we consider the concatenation $\beta_n = \gamma_n \star \eta_n \star \xi_n$. By definition $\beta_n \in \Gamma_{x,w_n}$, so $\beta_n \cap A \neq \emptyset$. Since $\gamma_n \cap A = \emptyset$ we can find a point $v_n$ of $A$ on the curve $\eta_n \star \xi_n$. By definition $\sfd(z,v_n) \leq \Lambda_z \max\{\sfd(z,z_n),\sfd(z,w_n)\} \to 0$ as $n$ goes to infinity. This implies that $z\in A$ since $A$ is closed.
\end{proof}

\subsection{The position function}

For us, the main object associated to a set is its position function.
For a given set $A \subseteq \X$, a Lipschitz curve $\gamma \in C([0,1],\X)$ and two parameters $s_1 \le s_2 \in [0,1]$, we define
\begin{equation}
\label{eq:defin_length_set}
    \ell(\gamma \cap A,s_1,s_2):= \ell(\gamma \restr{[s_1,s_2]} \cap A) = \int_{s_1}^{s_2} \chi_A(\gamma(t)) \vert \dot{\gamma}(t)\vert \,\d t.
\end{equation}
The second equality follows from the definition of $\gamma\restr{[s_1,s_2]}$. Also the next useful property follows directly from the definitions:
\begin{equation}
    %\label{eq:length_concatenation}
    \ell((\gamma \star \eta) \cap A, s_1, s_2) = \ell(\gamma \cap A, \min\{1,2s_1\}, \min\{1,2s_2\}) + \ell(\eta \cap A, \max\{0,2s_1 - 1\}, \max\{0,2s_2 - 1\}).
\end{equation}
%However the quantity does not change and we can use directly the expression in the right hand side. 
%Indeed let $\phi \colon [0,1] \to [0,1]$ be a Lipschitz increasing function such that $\phi(0)=0, \phi(1)=1$. For a given Lipschitz curve $\gamma \in C([0,1],\X)$, we define $\tilde{\gamma}(t):=\gamma(\phi(t))$ . Then
%\begin{equation}
%\label{eq:reparametrization_length_between_times}
%    \ell(\gamma \cap A, s, t) = \ell(\tilde{\gamma} \cap A, \phi^{-1}(s), \phi^{-1}(t)).
%\end{equation}
%There are other useful consequences of this fact. First of all
%\begin{equation}
%    \ell(\gamma \cap A, s, t) = \ell(\gamma\restr{[0,t] \cap A, }
%\end{equation}
%

\noindent For simplicity, we also define $\ell(\gamma \cap A,s):=\ell(\gamma \cap A,0,s)$. In this case the formula above reduces to
\begin{equation}
\label{eq:length_concatenation_from_0}
    \ell((\gamma \star \eta) \cap A, s) = \ell(\gamma \cap A, 0, \min\{1,2s\}) + \ell(\eta \cap A, 0, \max\{0,2s - 1\}).
\end{equation}
For $A \subseteq \X$ and $\gamma \in \Gamma_{x,y}$ we define the \emph{position function along the path $\gamma$ of the set $A$} as 
\begin{equation}
    \label{eq:posgamma_continuous_definition}
    \pos_{\gamma, A} \colon \X \to [0,\infty],\qquad \pos_{\gamma,A}(z):= \inf_{s \in \gamma^{-1}(z)} \ell(\gamma \cap A,s).
\end{equation}
with the usual convention that ${\rm pos}_{\gamma,A}(z)=\infty$ if $z\notin {\rm Im}(\gamma)$. In other words $\pos_{\gamma,A}(z)$ is the length spent by $\gamma$ inside $A$ before reaching for the first time the point $z$.
We define the \emph{position function with respect to the set $A$} as
\begin{equation}
\pos_A \colon \X \to [0,\infty],\qquad \pos_A(z):=\inf_{\gamma \in \Gamma_{x,y}} \pos_{\gamma,A}(z).
\end{equation}
Therefore the quantity $\pos_A(z)$ denotes the minimal length that a curve between $x$ and $y$ has to spend inside $A$ before passing through $z$.\\
The position function satisfies an interesting property: it allows to naturally fibrate a set into path-separating subsets by looking at its level sets. In order to prove it we introduce the following useful notation. Let $A$ and $\gamma\in \Gamma_{x,y}$, in particular $\gamma$ is Lipschitz. For $\tau \in [0,\ell(\gamma \cap A)]$, we define $(\gamma \cap A)_\tau \in \X$ as
\begin{equation}
    (\gamma \cap A)_\tau=\gamma(\bar{s}),\qquad \text{where }\bar{s}:=\min\{ s:\, \ell(\gamma \cap A, s) \ge \tau\}.
\end{equation}
We remark that, since $\gamma$ is Lipschitz then the function $s\mapsto \ell(\gamma\cap A, s)$ is continuous and the minimum above is well defined.
In the next proposition we will use the properties of the function $[0,\ell(\gamma \cap A)] \ni \tau \mapsto \pos_A((\gamma \cap A)_\tau) \in [0,\infty]$. Figure \ref{fig:pos_along_curve} illustrates the behaviour of this function in a specific example.
\begin{figure}[h]
    \centering
    \includegraphics[scale=1.3]{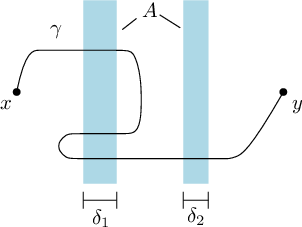}
    \includegraphics[scale=0.7]{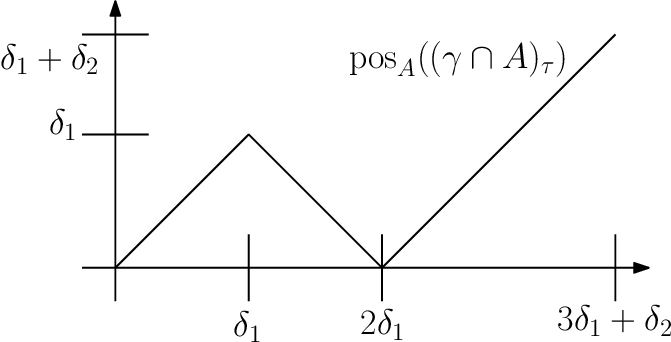}
    \caption{Let us consider $\gamma$ and $A \subseteq \R^2$ as in the picture. In such a case the width of $A$ with respect to $x,y$ is equal to $\delta_1+\delta_2$ and in particular is realized by the straight curve connecting $x$ to $y$. The function ${\rm pos}_A((\gamma \cap A)_\tau)$ computes the position of the point reached when the curve travels a piece of length $\tau$ inside $A$. Its graph is reported in the picture. In this case the level set of the position function $\pos_A$ fibrates the set $A$ by straight vertical lines.}
    \label{fig:pos_along_curve}
\end{figure}

\begin{proposition}
\label{prop:level_set_position_function}
        Let $(\X,\sfd)$ be a metric space and let $x,y\in \X$. Let $A\in \SS_{\textup{path}}(x,y)$ be closed. Then the sets $\{\pos_A = t \} \cap A$ belong to $\SS_{\textup{path}}(x,y)$ for all $t \in [0,\width_{x,y}(A)]$.
\end{proposition}
We will see in the proof that the assumption $A\in \SS_{\textup{path}}(x,y)$ is needed in order to guarantee that $\{\pos_A = t \} \cap A$ is not empty for all $t\in [0,\width_{x,y}(A)]$. The assumption that $A$ is closed is made in order to simplify the proof. A similar statement holds for every separating set.
\begin{proof}
    If $\Gamma_{x,y} = \emptyset$ then every subset of $\X$ is a path separating set, thus the statement is trivial. So we can assume that $\Gamma_{x,y} \neq \emptyset$.
    Let $A \in \SS_\text{path}(x,y)$ be closed and $\pos_A \colon \X \to [0,+\infty]$ be its position function.
    \vspace{1mm}
    
    \noindent\textbf{Claim 1:} For every $\gamma \in \Gamma_{x,y}$ there exists $z \in \gamma \cap A$ such that $\pos_A(z) \ge \width_{x,y}(A)$. \\
    We define $s_{\max} :=\max\{ s :\, s \in \gamma^{-1}(A) \}$, which is well defined since $\gamma \cap A \neq \emptyset$ and $A$ is closed. We prove that the point $z:=\gamma(s_{\max})$ verifies the claim. By contradiction we suppose that $\pos_A(z) < \width_{x,y}(A)$. For every $\varepsilon >0$ we consider a path $\gamma_\varepsilon \in \Gamma_{x,y}$ such that $\pos_{\gamma_\varepsilon, A}(z) \le \pos_A(z) + \varepsilon$. By definition $z = \gamma_\varepsilon(s_\varepsilon)$ for some $s_\varepsilon$ such that 
    \begin{equation}
    \label{eq:optimality_choice_teps}
        \ell(\gamma_\varepsilon \cap A, s_\varepsilon) < \pos_{\gamma_\varepsilon, A}(z)+\varepsilon.
    \end{equation}
    We define $\eta_\varepsilon = \gamma_\varepsilon\restr{[0,s_\varepsilon]} \star \gamma\restr{[s_{\max},1]} \in \Gamma_{x,y}$.
    We compute, using \eqref{eq:defin_length_set} and \eqref{eq:length_concatenation_from_0},
    \begin{equation}
    %\begin{aligned}
        \ell(\eta_\varepsilon \cap A) = \ell(\eta_\varepsilon \cap A, 0, 1) = \ell(\gamma_\varepsilon\restr{[0,s_\varepsilon]} \cap A) + \ell(\gamma \restr{[s_{\max},1]} \cap A) 
        %&= \ell(\gamma_\varepsilon \cap A, 0,s_\varepsilon) + \ell(\gamma \cap A, s_{\max},1])\\
        \le \pos_A(z) + 2\varepsilon,
    %\end{aligned}
    \end{equation}
    because $\ell(\gamma \restr{[s_{\max},1]} \cap A) = 0$ by definition of $s_{\max}$. Therefore 
    \begin{equation}
    \label{eq:chain_inequality_step_1}
        \width_{x,y}(A) \leq \ell(\eta_\varepsilon \cap A) \le \pos_A(z) + 2\varepsilon.
    \end{equation}
    If $\varepsilon$ is chosen to be sufficiently small, we have that $\pos_A(z) + 2\varepsilon < \width_{x,y}(A)$, thus giving a contradiction.
    \vspace{1mm}
    
    \noindent\textbf{Claim 2:} for every $\gamma \in \Gamma_{x,y}$ we have
    \begin{equation}
    \label{eq:one_lipschitzianity_position}
        \vert \pos_A((\gamma \cap A)_{\tau_1}) -  \pos_A((\gamma \cap A)_{\tau_2}) \vert \leq \vert \tau_1-\tau_2 \vert,\,\text{for }0 \le \tau_1,\tau_2 \le \ell(\gamma \cap A).
    \end{equation}
    We fix $\tau_1,\tau_2 \in [0,\ell(\gamma \cap A)]$ and we assume that $\tau_1 \ge \tau_2$. By definition $(\gamma \cap A)_{\tau_1} = \gamma({\bar{s}_1})$ and $(\gamma \cap A)_{\tau_2} = \gamma({\bar{s}_2})$ for some ${\bar{s}_1} \ge {\bar{s}_2} \in [0,1]$. We split the claim in two parts. We first prove that $\pos_A((\gamma \cap A)_{\tau_1}) -  \pos_A((\gamma \cap A)_{\tau_2}) \leq \tau_1 - \tau_2$.
    For every $\varepsilon$, we consider $\eta_{\varepsilon} \in \Gamma_{x,y}$ and $s_\varepsilon\in [0,1]$ such that $(\gamma \cap A)_{\tau_2} = \eta_{\varepsilon}(s_\varepsilon)$ and $$\ell(\eta_{\varepsilon} \cap A, s_\varepsilon) \leq \pos_A((\gamma \cap A)_{\tau_2}) + \varepsilon.$$
    We define $\gamma_\varepsilon = \eta_\varepsilon \restr{[0,s_\varepsilon]} \star \gamma\restr{[\bar{s}_2,1]} \in \Gamma_{x,y}$. We now define $\bar{t}_1 := \frac{\bar{s}_1 - \bar{s}_2}{1 - \bar{s}_2}$. It satisfies $\gamma\restr{[\bar{s}_2,1]}(\bar{t}_1) = \gamma(\bar{s}_1) = (\gamma \cap A)_{\tau_1}$. Finally if we set $\tilde{t}_1 = \frac{1 + \bar{t}_1}{2}$ we get $\gamma_\varepsilon(\tilde{t}_1) = (\gamma \cap A)_{\tau_1}$.
    Now we compute, using again \eqref{eq:length_concatenation_from_0} and the definition of the points $(\gamma \cap A)_{\tau_1}, (\gamma \cap A)_{\tau_2}$,
    \begin{equation}
        \begin{aligned}
            \pos_{A}((\gamma \cap A)_{\tau_1}) \le \ell(\gamma_\varepsilon \cap A, 0, \tilde{t}_1) &= \ell(\eta_\varepsilon \cap A, 0, s_\varepsilon) + \ell(\gamma \cap A, \bar{s}_2, \bar{s}_1)\\
            &\le \pos_A((\gamma\cap A)_{\tau_2}) + \varepsilon + \tau_1 - \tau_2.
        \end{aligned}
    \end{equation}
    By taking the limit as $\varepsilon \to 0$, we conclude that $\pos_A((\gamma \cap A)_{\tau_1}) -  \pos_A((\gamma \cap A)_{\tau_2}) \leq \tau_1 - \tau_2$.\\
    For the other inequality the procedure is similar and we just sketch it. For every $\varepsilon > 0$ we consider $\eta_\varepsilon \in \Gamma_{x,y}$ and $s_\varepsilon \in [0,1]$ such that $(\gamma \cap A)_{\tau_1} = \eta_\varepsilon(s_\varepsilon)$ and 
    $$\ell(\eta_{\varepsilon} \cap A, s_\varepsilon) \leq \pos_A((\gamma \cap A)_{\tau_1}) + \varepsilon.$$
    This time we define the curve $\gamma_\varepsilon = \eta_\varepsilon \star (-\gamma \restr{[\bar{s}_2, \bar{s}_1]}) \star \gamma\restr{[\bar{s}_2,1]} \in \Gamma_{x,y}$ and arguing as above we arrive to
    \begin{equation}
        \begin{aligned}
            \pos_{A}((\gamma \cap A)_{\tau_2}) \le \pos_A((\gamma\cap A)_{\tau_1}) + \varepsilon + \tau_1 - \tau_2.
        \end{aligned}
    \end{equation}
    By taking $\varepsilon \to 0$ we conclude the proof of the claim.

    %We define $\gamma_{\varepsilon}$ as the concatenation of $\eta_{\varepsilon}$ up to $s_\varepsilon$ and $\gamma$ , suitably rescaled from $[s_0,1]$ to $[s_\varepsilon,1]$. We define the time $\tau^* := \inf \{\tau \geq s_\varepsilon \, : \, \gamma_{\varepsilon}(\tau) = (\gamma \cap A)_t  \}$. We also consider $t_0$ such that $(\gamma \cap A)_t = \gamma_{t_0}$.   
    %By applying \eqref{eq:reparametrization_length_between_times} and since $\eta_\varepsilon =\gamma_\epsilon$ on $[0,s_\varepsilon]$, we have respectively that
    %\begin{equation}
    %\label{eq:properties_concatenation_claim2}
    %    \ell(\gamma \cap A, s_0,t_0) = \ell(\gamma_{\varepsilon} \cap A, s_\varepsilon,\tau^*)\qquad\text{and} \qquad\ell(\gamma_{\varepsilon} \cap A,0,s_\varepsilon)= \ell(\eta_{\varepsilon} \cap A,0,s_\varepsilon).
    %\end{equation}
    %Therefore we have
    %\begin{equation*}
    %\begin{aligned}
    %    \pos_{A}((\gamma \cap A)_t)  \le \pos_{\gamma_{\varepsilon},A}((\gamma \cap A)_t)        
    %    \le \ell(\gamma_{\varepsilon} \cap A, \tau^*) 
     %   &\le \ell(\gamma_{\varepsilon} \cap A,0,s_\varepsilon) + \ell(\gamma_{\varepsilon} \cap A, s_\varepsilon,\tau^*)\\
%&\stackrel{\eqref{eq:properties_concatenation_claim2}}{\le} \ell(\eta_{\varepsilon} \cap A, s_\varepsilon) + \ell(\gamma \cap A, s_0,t_0)\\
 %       &\le \pos_{\eta_{\varepsilon}, A}((\gamma \cap A)_s) + \varepsilon + \ell(\gamma \cap A, s_0,t_0) \\
  %      & \le \pos_A((\gamma \cap A)_s) + t-s +2\varepsilon.\\
   % \end{aligned}
    %\end{equation*}
\vspace{1mm}
    
    \noindent\textbf{Claim 3:} $\{ \pos_A = t\} \in \SS_{\text{path}}(x,y)$ for every $t \in [0,\width_{x,y}(A)]$.\\
    Let us fix $t_0 \in [0,\width_{x,y}(A)]$. We have to prove that $\gamma \cap \{ \pos_A = t_0 \}\neq \emptyset$ for every $\gamma \in \Gamma_{x,y}$. If $t_0=0$ the claim is true because every $\gamma \in \Gamma_{x,y}$ intersects $A$, since $A\in \SS_\text{path}(x,y)$ and $(\gamma \cap A)_0$ satisfies $\pos_A((\gamma\cap A)_0) = 0$. 
    The function
    \begin{equation*}
        [0,\ell(\gamma \cap A)] \ni \tau \mapsto \pos_A((\gamma \cap A)_\tau) \in [0,\infty)
    \end{equation*}
    is $1$-Lipschitz by Claim 2. Moreover, by Claim 1, there exists a point $z \in \gamma \cap A$ such that $\pos_A(z) \ge \width_{x,y}(A)$, in particular there exists $\tau_{\max} \in [0,\ell(\gamma \cap A)]$ such that $\pos_A((\gamma \cap A)_{\tau_{\max}}) \ge \width_{x,y}(A)$. Thus, by mean value theorem, there exists a $\tau\in [0,\tau_{\max}]$ such that $\pos_A((\gamma \cap A)_{\tau}) = t_0$.
\end{proof}
In general we cannot use this decomposition unless we know more properties of the position function of separating sets. The next proposition shows that it is more regular under some connectedness properties of the space. 
\begin{proposition}
\label{prop:position_local_properties}
        Let $(\X,\sfd)$ be a rectifiable path connected metric space and let $x,y\in \X$. Let $A \subseteq \X$. Then 
        \begin{itemize}
            \item[(i)] $\pos_A(z) < +\infty$ for every $z\in \X$;
            \item[(ii)] if $\X$ is pointwise $\Lambda_z$-quasiconvex at $z\in \X$ then $\lip(\pos_A)(z) \leq \Lambda_z$;
            \item[(iii)] if $\X$ is pointwise quasiconvex then $\pos_A$ is continuous and $\lip(\pos_A) = 0$ on $\overline{A}^c$;
            \item[(iv)] if $\X$ is (locally) $\Lambda$-quasiconvex then $\pos_A$ is (locally) $\Lambda$-Lipschitz.
        \end{itemize}
\end{proposition}
\begin{proof}
    The first assertion is trivial, so we move to the second one. Let us fix $z\in \X$ and $r_z>0$ such that $\Gamma_{z,w}^{\Lambda_z} \neq \emptyset$ for every point $w \in B_{r_z}(z)$. Let us fix $w \in B_{r_z}(z)$ and $\varepsilon > 0$. Let $\eta_\varepsilon \in \Gamma_{x,y}$ and $s_\varepsilon \in [0,1]$ be such that $z=\eta_\varepsilon(s_\varepsilon)$ and $\pos_A(z) \leq \ell(\eta_\varepsilon \cap A, s_\varepsilon) + \varepsilon$. We also take a curve $\gamma_w \in \Gamma_{z,w}^{\Lambda_z}$. We define the curve $\gamma = \eta_\varepsilon\restr{[0,s_\varepsilon]} \star \gamma_w \star (-\gamma_w) \star \eta_\varepsilon\restr{[s_\varepsilon,1]} \in \Gamma_{x,y}$. By the usual application of \eqref{eq:length_concatenation_from_0} we get
     %We denote by $\bar{\gamma}_w$ the curve $\gamma_w$ with inverse orientation and we define $\gamma$ as the concatenation of $\eta_\varepsilon$ up to $t_\varepsilon$, $\gamma_w$, $\bar{\gamma}_w$ and $\eta_\varepsilon$ after $t_\varepsilon$. We have $\gamma \in \Gamma_{x,y}$ and $w = \gamma(\tau)$ where $\tau = \min\{s\geq t_\varepsilon\,:\, \gamma(s)=w\}$. Then we can estimate
    $$\pos_A(w) \leq \ell(\eta_\varepsilon \cap A, s_\varepsilon) + \Lambda_z\sfd(z,w) \leq \pos_A(z) + \Lambda_z\sfd(z,w) + \varepsilon,$$
    and, by the arbitrariness of $\varepsilon$, 
    $$\pos_A(w) \leq \pos_A(z) + \Lambda_z\sfd(z,w).$$
    Similarly we also obtain 
    $$\pos_A(z) \leq \pos_A(w) + \Lambda_z\sfd(z,w).$$
    Therefore we can prove (ii) since
    $$\lip(\pos_A)(z) = \lim_{r\to 0} \sup_{w\in B_r(z)} \frac{\vert\pos_A(w) - \pos_A(z)\vert}{\sfd(z,w)} \leq \Lambda_z.$$
    The same technique shows (iv). The continuity statement in (iii) can be proved exactly as in \cite[Lemma 2.2]{Durand-Cartagena-Jaramillo}. It remains to show that $\lip(\pos_A)(z)=0$ if $z\in \overline{A}^c$, in case $\X$ is pointwise quasiconvex. Let $r_z,\Lambda_z$ be as in the definition of pointwise quasiconvexity at $z$. For every $\varepsilon > 0$ let $\gamma_\varepsilon \in \Gamma_{x,y}$ and $s_\varepsilon$ such that $z = \gamma_\varepsilon(s_\varepsilon)$ and $\ell(\gamma \cap A, s_\varepsilon) \le \pos_A(z) + \varepsilon$. Take $r<\min\{r_z, \frac{\sfd(z,\overline{A})}{\Lambda_z}\}$. For every $w\in B_r(z)$ take a curve $\gamma_{w} \in \Gamma_{z,w}^{\Lambda_z}$. By construction $\gamma_{w}$ avoids $A$ and so, arguing as in the first part of the proof, we deduce that $\pos_A(z) = \pos_A(w)$. Since this is true for every $w\in B_r(z)$ we conclude that $\lip(\pos_A)(z) = 0$.
\end{proof}

\subsection{Equivalence of minima via the position function}
\label{subsec:equivalence_minima}
In this section we use the properties of the position function of a path separating set to show the equivalence of the minima of the separating ratio and the Minkowski content. Observe that this result, especially in the locally geodesic setting, sharpens Theorem \ref{theo:main-intro-p=1}. If we apply it to the measure $\mm_{x,y}^L$, where $\mm$ is doubling, we get Theorem \ref{thm:main_equivalence_minima}.

\begin{theorem}
\label{thm:inf_TOP_PATH}
    Let $(\X,\sfd,\mm)$ be a path connected, locally $\Lambda$-quasiconvex metric measure space and suppose $\supp(\mm)$ is compact. Let $x,y \in \supp(\mm)$. Then
\begin{equation}
    \label{eq:equivalence_of_minima}
    \Lambda^{-1}\inf_{\Omega \in \SS_{\textup{top}}(x,y)} \mm^+(\Omega) \le \inf_{A \subseteq \X\,\textup{closed}} \frac{\mm(A)}{\width_{x,y}(A)}\le \inf_{\Omega \in \SS_{\textup{top}}(x,y)} \mm^+(\Omega).
\end{equation}
In particular, if $(\X,\sfd)$ is path connected and locally geodesic we have
\begin{equation}
    \label{eq:equality_of_minima_geodesic_case}
    \inf_{\Omega \in \SS_{\textup{top}}(x,y)} \mm^+(\Omega) = \inf_{A \subseteq \X\,\textup{closed}} \frac{\mm(A)}{\width_{x,y}(A)}.
\end{equation}
\end{theorem}

\begin{proof}
    We reproduce here the proof of the right inequality. Let $\Omega \in \SS_{\textup{top}}(x,y)$. We can suppose $\mm^+(\Omega) < +\infty$ otherwise there is nothing to prove. In particular we have $\mm^+(\Omega) < +\infty$, which implies $\mm(\partial \Omega) = 0$. We consider the sets $A_r:= \overline{B}_r(\Omega) \setminus \Int(\Omega)$, which is a closed subset of $\X$. As in the proof of Theorem \ref{theo:main-intro-p=1} we get $\width_{x,y}(A_r) \geq r$ if $r < \min\{\sfd(\partial \Omega, x), \sfd(\partial \Omega, y) \}$. Therefore
     $$\mm^+(\Omega) = \limi_{r \to 0}\frac{\mm(\overline{B}_r(\Omega) \setminus \Omega)}{r} = \limi_{r \to 0}\frac{\mm(A_r)}{r} \geq \limi_{r \to 0}\frac{\mm(A_r)}{\width_{x,y}(A_r)} \geq \inf_{\substack{A \subseteq \X \\ A \text{ closed}}} \SR_{x,y}(A).$$
    In order to show the first inequality we proceed as follows. First of all notice that $\Gamma_{x,y} \neq \emptyset$ because $(\X,\sfd)$ is rectifiable path connected, so $\width_{x,y}(A) < +\infty$ for every $A\subseteq \X$. Now we fix an arbitrary closed subset $A \subseteq \X$. We can suppose $\width_{x,y}(A) > 0$, otherwise $\SR_{x,y}(A) = +\infty$ and there is nothing to prove. In particular $A \in \SS_{\text{path}}(x,y)$. Let $\pos_A$ be the position function associated to $A$. It is locally $\Lambda$-Lipschitz by Proposition \ref{prop:position_local_properties}. Therefore it is Lipschitz when restricted to the compact set $\supp(\mm)$ because of Lemma \ref{lemma:loclip_to_lip_up_to_scale}. By the coarea inequality of Proposition \ref{prop:coarea_inequality_minkowski} we get
    $$\int_0^{\width_{x,y}(A)} \mm^+(\{\pos_A \le t\})\,\d t \leq \int_\X\lip(\pos_A)\,\d \mm.$$
    We notice that, for $t\in (0,\width_{x,y}(A))$, the sets $\{\pos_A \le t\}$ belong to $\SS_{\text{top}}(x,y)$ because $\pos_A$ is Lipschitz on $\supp(\mm)$, $x,y\in \supp(\mm)$ and $\pos_A(x) = 0$, $\pos_A(y) = \width_{x,y}(A)$. Moreover $\lip(\pos_A) = 0$ on $A^c$ since $A$ is closed, by Proposition \ref{prop:position_local_properties}. Hence, again by Proposition \ref{prop:position_local_properties}, we have
    $$\int_0^{\width_{x,y}(A)} \mm^+(\{\pos_A \le t\})\,\d t \leq \int_A\lip(\pos_A)\,\d \mm \le \Lambda \mm(A).$$
    Dividing by $\width_{x,y}(A)$ we conclude that we can find some $t\in (0,\width_{x,y}(A))$ such that
    $$\mm^+(\{\pos_A \le t\} \le \Lambda \frac{\mm(A)}{\width_{x,y}(A)} = \Lambda\cdot \SR_{x,y}(A).$$
    By the arbitrariness of $A$ we finally have
    $$\frac{1}{\Lambda}\inf_{\Omega \in \SS_{\textup{top}}(x,y)} \mm^+(\Omega) \le \inf_{A \subseteq \X\,\textup{closed}} \frac{\mm(A)}{\width_{x,y}(A)}.$$
\end{proof}
The proof of Theorem \ref{thm:inf_TOP_PATH} relies on the coarea inequality for the Minkowski content and a special role is played by the position function. In the next section we will show an analogous statement in case of discrete metric measure spaces, where one can use directly the properties of the discrete analogous of the position function. We believe that the proof in the discrete setting highlights the main geometric intuitions behind the equivalence of the minima of the two functionals. 

\section{The equivalence of the infima on measure graphs}

\label{sec:graph_theory}
Let us consider a graph, i.e.\ a couple $(\V,E)$, where $\V$ is a set, called the \emph{set of vertices} and $E\subseteq (\V \times \V) / \mathfrak{S}_2$ is called the \emph{set of edges}. $\mathfrak{S}_2$ is the permutation group of order 2 and its non-trivial element acts on $\V\times \V$ as $(v,w) \mapsto (w,v)$. We say that two points $v,w$ are adjacent, and we write $v \sim w$, if $(v,w) \in E$. 
%The graph $(\V,E)$ is locally finite if $\# \{ w \in \V : v \sim w\} < \infty$ for every $v \in \V$. 
We say that $(\V,E)$ is a countable graph if $\V$ is countable.
A path in $\V$ is an element of the set
\begin{equation*}
    \mathcal{P}:=\left\{ \{q_i\}_{i=0}^N:N \in \mathbb{N},\,q_i \in \V,\, (q_i,q_{i+1})\in E \right\}.
\end{equation*}
The set of paths connecting two points $v,w \in \V$ is
\begin{equation*}
    \mathcal{P}_{v,w}:=\left\{ \{q_i\}_{i=0}^N \in \mathcal{P} : q_0 = v, q_N = w \right\}.
\end{equation*}
The concatenation of two paths $\sfc = \{q_i \}_{i=0}^N \in \mathcal{P}_{v,w}$ and $\sfc' = \{ q_i'\}_{i=0}^{N'} \in \mathcal{P}_{w,z}$ is the path 
$$\sfc \star \sfc' = (q_0,\ldots, q_N = q_0', q_1',\ldots, q_{N'}) \in \mathcal{P}_{v,z}.$$
Given $\sfc \in \mathcal{P}$ and a subset $A\subseteq \X$, the set $(\sfc \cap A)$ is naturally ordered by the order of elements of $\sfc$. Given $j\in \N$, we denote by $(\sfc \cap A)_j$ the $j$-th element of $(\sfc \cap A)$ with respect to this order.

%Let $(V,E)$ be a graph and consider a function $\ell \colon E \to (0,+\infty)$. We define 
%$$\rho(v,w) := \inf_{\{q_i\}_{i=1}^N \in \mathcal{P}_{v,w}} \sum_{i=1}^N \ell((q_i, q_{i+1})),$$
%which is a distance on $V$.
%Given a countable, locally finite graph $(V,E)$, $\mu$ and $\ell$ as above, we call the triple $(V,\rho,\mu)$ a metric measure graph.

\noindent We define the analogous version in the discrete setting of the width.

\begin{definition}[Discrete width and discrete separating sets]
    Let $(\V, E)$ be a graph and let $v,w \in \V$.
    The width with respect to $v,w$ of a subset $A \subseteq \V$ is 
    $$\dwidth_{v,w}(A) := \inf\{ \# (\sfc \cap A) : \sfc \in \mathcal{P}_{v,w} \}$$
    with the convention that $\dwidth_{v,w}(A) = 0$ if $\mathcal{P}_{v,w} = \emptyset$.
    A set $A\subseteq V$ is a separating set between $v$ and $w$ if $\dwidth_{v,w}(A) \ge 1$, i.e. if every path between $v$ and $w$ intersects $A$. The class of separating sets between $v$ and $w$ is denoted by $\dSS(v,w)$. For instance if $\mathcal{P}_{v,w} = \emptyset$ then every subset of $\V$ is separating, in particular $\emptyset \in \mathcal{P}_{v,w}$.
\end{definition}

Given a countable graph $(\V,E)$, we consider a set function $\mu \colon \V \to (0,+\infty)$. We extend $\mu$ to a measure on $2^\V$ in the natural way by summing up over the singletons. A measure graph is a triple $(\V,E,\mu)$ as above. The analogous of the separating ratio is defined as follows.
    
\begin{definition}[Discrete separating ratio]
    Let $(\V, E, \mu)$ be a measure graph and let $v,w \in \V$. The discrete separating ratio of a subset $A \subseteq V$ with respect to $v,w$ is
    $$\dSR_{v,w}(A):=\frac{\mu(A)}{\dwidth_{v,w}(A)}$$ 
    with the convention that $\dSR_{v,w}(A) = 0$ if $\dwidth_{v,w}(A) = +\infty$, i.e. if $\mathcal{P}_{v,w} = \emptyset$, and $\dSR_{v,w}(A) = +\infty$ if $\dwidth_{v,w}(A) = 0$, i.e. if $A \notin \dSS(v,w)$.
\end{definition}

\begin{comment}
\begin{definition}[Discrete Minkowski content]
    Let $(\V, E, \mu)$ be a measure graph and let $A\subseteq \X$. The enlargement of $A$ is the set $$\textup{enl}(A) := \{ v \in \V\,:\, \exists w \in A \, : \, v \sim w \}.$$
    The discrete Minkowski content of $A$ is 
    $$\mu^+(A) := \mu(\textup{enl}(A) \setminus A).$$
\end{definition}
\end{comment}

\begin{remark}
    We first remark that our construction and our result do not require a metric on a graph; in this way, we highlight the combinatorial aspect of the theory in this discrete setting.
\end{remark}

\begin{remark}
\label{rmk:graph_approximation}
    A classical example of a graph is given by an approximation of a metric space in terms of nets.
    Let $(\X,\sfd)$ be a metric space. For every $r > 0$, let $\X_{r}$ be a maximal $r$-separated subset of $\X$: this means that the distance of every two points of $\X_r$ is at least $r$ and that $\X_r$ is maximal with this property. It is standard that such a set is $2r$-dense, i.e. every point of $\X$ is at distance at most $2r$ from a point of $\X_r$. Let us construct the graph $(\V_r, E_r)$ in the following way: $\V_r = \X_r$ as a set and $v\sim w$ if and only if $\sfd(v,w)\le 3r$. This discretization procedure of a metric space is standard. \\
    Assume now that $(\X,\sfd, \mm)$ is a doubling metric measure space and let $(\V_r,E_r)$ be as above. Now consider $\mu_r \colon \V_r \to (0,+\infty)$ defined by $\mu_r(v) := \mm(B_{3r}(v))$. This defines a measure graph $(\V_r, E_r, \mu_r)$. For every $A\subseteq \X$ one can define the set $A_r := \{ v \in \V_r \, : \, \sfd(v,A) < 2r \}$. Observe that, as subsets of $\X$, we have $A_r \subseteq B_{2r}(A)$ and $A \subseteq B_{2r}(A_r)$. In particular 
    \begin{equation*}
        \mm(B_r(A)) \le \mm(B_{3r}(A_r)) \le \sum_{v\in A_r} \mm(B_{3r}(v)) = \mu_r(A_r) \le C_D^2 \sum_{v\in A_r} \mm(B_{r}(v)) \le C_D^2\mm(B_{3r}(A)).
    \end{equation*}
    Therefore, if for simplicity we restrict to the sets with $\mm(A) = 0$, the last computation implies that 
    $$\frac{\mu_r(A_r)}{r} \approx \mm^+(A)$$
    as $r$ goes to $0$. Similarly, at least if $(\X,\sfd)$ is $\Lambda$-quasigeodesic, we have
    $$r\cdot \dwidth_{x,y}(A_r) \approx \width_{x,y}(A).$$
    Those are qualitative relations between the quantities defined in the discrete and the continuous settings. We will discuss the construction also later in Remark \ref{rem:gamma_convergence}. 
\end{remark}

In view of the observations above, it is clear that the next result is the discrete analogous of Theorem \ref{thm:inf_TOP_PATH}.

\begin{theorem}
\label{thm:equivalence_minimum_energies_discrete}
    Let $(\V, E, \mu)$ be a measure graph and let $v,w \in \V$. Then:
    $$\inf_{A \subseteq \V} \dSR_{v,w}(A) = \inf_{A \in \dSS(v,w)}\mu(A).$$
\end{theorem}

The proof uses the analogous of the position function in the discrete setting. Let $A \subseteq \X$ and $\sfc = \{q_i\}_{i=0}^N \in \mathcal{P}_{v,w}$ be a path. The position function with respect to the set $A$ and the path $\sfc$ is the function ${\rm pos}_{\sfc,A} \colon \V \to \N 
 \cup \{\infty \}$ defined by 
 $$\pos_{\sfc,A}(z) := \min_{i_0\in \N\,:\, z = q_{i_0}} \#(\{q_i\}_{i=0}^{i_0} \cap A).$$ 
If $z \notin \sfc$ then $\pos_{\sfc,A}(z) = +\infty$. The position function with respect of the set $A$ is 
$$\pos_A \colon \V \to \mathbb{N} \cup \{\infty\}, \quad z\mapsto \min_{\sfc \in \mathcal{P}_{v,w}} \pos_{\sfc,A}(z).$$

\noindent For instance $\pos_A(z) = 1$ if there exists $\sfc \in \mathcal{P}_{v,w}$ such that $z$ is the first point of intersection of $\sfc$ with $A$. Figure \ref{fig:position function} gives an intuition about this definition.

\begin{figure}[h!]
    \centering
    \includegraphics[scale=1.8]{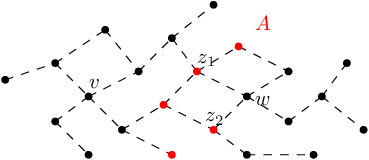}\\
    \vspace{5mm}
    \includegraphics[scale=1.8]{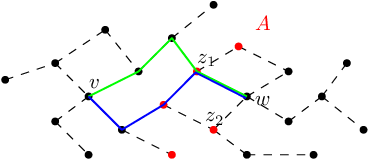}    
    \caption{For the given set $A$ made of red points, the point $z_1$ has position $1$, i.e.\ ${\pos_A}(z_1)=1$, while $\pos_A(z_2)=2$. Indeed $z_1$ has position $2$ with respect to the blue path, but it has position $1$ with respect to the green one. Instead there are no paths for which $z_2$ is the first intersection point with $A$.}
    \label{fig:position function}
\end{figure}

\noindent The analogous of Proposition \ref{prop:level_set_position_function} allows to fibrate a set $A$ by means of level sets of the position function associated to $A$. Figure \ref{fig:fibration} shows such a procedure for the graph and the set $A$ in Figure \ref{fig:position function}. This is is the key step in the proof of Theorem \ref{thm:equivalence_minimum_energies_discrete}. Notice that, unlike the continuous case, we do not need additional properties of the position function related to the connectivity properties of the metric space.

\begin{figure}[h!]
    \centering
    \includegraphics[scale=1.8]{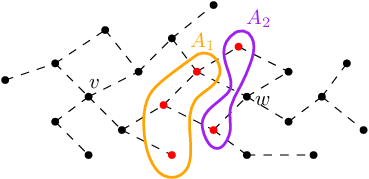}
    \caption{The construction of the sets $A_i$'s as in the proof.}
    \label{fig:fibration}
\end{figure}

\begin{proof}[Proof of Theorem \ref{thm:equivalence_minimum_energies_discrete}]
    If $\mathcal{P}_{v,w} = \emptyset$ then the thesis is true. Indeed for every subset of $A$ we have $\dSR_{v,w}(A) = 0$ by definition. On the other hand $\emptyset \in \dSS(v,w)$, so also the right hand side is zero. We can now suppose $\mathcal{P}_{v,w} \neq \emptyset$.    
    The central step is to show that
    \begin{equation}
        \label{eq:inf_discwidth_1}
        \inf_{A \subseteq \V} \dSR_{v,w}(A) = \inf_{\substack{A \subseteq \V \\ \dwidth_{v,w}(A) = 1}} \dSR_{v,w}(A).
    \end{equation}
    Indeed, if \eqref{eq:inf_discwidth_1} holds, then the thesis follows by the definition of $\dSR_{v,w}(A)$, since 
    \begin{equation*}
      \dSR_{v,w}(A) = \mu(A)\qquad\text{ if }\dwidth_{v,w}(A) = 1. 
    \end{equation*}
    It is clear that the right hand side of \eqref{eq:inf_discwidth_1} is always bigger than or equal to the left hand side, so it is enough to show the other inequality.\\
    Let us take $A\subseteq \V$. If $A\notin \dSS(v,w)$ then $\dSR_{v,w}(A) = +\infty$ and there is nothing to prove. Hence we can suppose that $A\in \dSS(v,w)$. It is enough to prove that there exists $A' \in \dSS(v,w)$ with $\dwidth_{v,w}(A') =1$ such that $\dSR(A) \ge \dSR(A')$.
    %If $\dwidth_{v,w}(A) = 1$ we just take $A' = A$, so in the sequel we suppose $\dwidth_{v,w}(A) \geq 2$.
    We consider the sets
    $$A_i := \lbrace z\in A \text{ s.t. } \pos_A(z) = i \rbrace.$$
    Similarly to what we did in the proof of Proposition \ref{prop:level_set_position_function} we claim that $A_i \in \dSS(v,w)$ and that $\dwidth_{v,w}(A_i)=1$ for $i=1,\ldots,\dwidth_{v,w}(A)$. 
    %It is clear from the definition, and the fact that $\mathcal{P}_{v,w} \neq \emptyset$, that $A_1 \in \dSR(v.w)$.\\
    \vspace{1mm}
    
    \noindent\textbf{Claim 1.} For every $\sfc \in \mathcal{P}_{v,w}$ there exists $z\in \sfc \cap A$ such that $\pos_A(z) \geq \dwidth_{v,w}(A)$. \\
    Let $z$ be the last point of the set $\sfc \cap A$. Suppose $\pos(z) < \dwidth_{v,w}(A)$ and take a path $\sfc_z\in \mathcal{P}_{v,w}$ such that $\pos_{\sfc_z,A}(z) = \pos_A(z)$. Consider the concatenation of $\sfc_z$ up to $z$ and $\sfc$ after $z$. This is a path belonging to $\mathcal{P}_{v,w}$ that intersects $A$ in at most $\pos_A(z) < \dwidth_{v,w}(A)$ points, a contradiction.
    \vspace{1mm}
    
    \noindent\textbf{Claim 2.} For every $\sfc \in \mathcal{P}_{v,w}$ we have $\pos((\sfc \cap A)_{j+1}) \leq \pos((\sfc \cap A)_j) + 1$ for every $j$.\\
    Call $z := (\sfc \cap A)_j$ and $p:= \pos((\sfc \cap A)_{j})$. Let $\sfc_z \in \mathcal{P}_{v,w}$ be such that $z = (\sfc_z \cap A)_p$. Consider the concatenation of $\sfc_z$ up to $z$ and $\sfc$ after $z$. This new curve belongs to $\mathcal{P}_{v,w}$ and $\pos((\sfc \cap A)_{j+1}) \leq p+1$ by construction.
    \vspace{1mm}
    
    \noindent\textbf{Claim 3.} $A_i \in \dSS(v,w)$ for all $i = 1,\ldots,\dwidth_{v,w}(A)$.\\
    Let $\sfc \in \mathcal{P}_{v,w}$. Since $\dwidth_{v,w}(A) \ge 1$ then the set $\sfc \cap A$ is not empty. The first point of intersection $(\sfc \cap A)_1$ clearly satisfies $\pos_A((\sfc \cap A)_1) = 1$. 
    By Claim 1 we can find a point $z\in \sfc$ with $\pos_A(z) \geq \dwidth_{v,w}(A)$. The function $\pos_A$ can increase along consecutive points of $\sfc \cap A$ by at most $1$ by Claim 2. So for every $i = 1,\ldots, \dwidth_{v,w}(A)$ there must be some point of $\sfc \cap A$ with position $i$, i.e. $\sfc \cap A_i \neq \emptyset.$
    \vspace{1mm}
    
    \noindent\textbf{Claim 4.} $\dwidth_{v,w}(A_i) = 1$ for all $i = 1,\ldots,\dwidth_{v,w}(A)$.\\
    Let $\sfc \in \mathcal{P}_{v,w}$ be arbitrary and $i \in \{ 1,\ldots, \dwidth_{v,w}(A)\}$. Call $z$ the last point of intersection of $\sfc \cap A_i$: it exists because of Claim 3. Let $\sfc_z \in \mathcal{P}_{v,w}$ be such that $z = (\sfc_z \cap A)_i$. Consider the concatenation of $\sfc_z$ up to $z$ and $\sfc$ after $z$. This is a path belonging to $\mathcal{P}_{v,w}$ that intersects $A_i$ only in $z$. Indeed in the part of $\sfc$ after $z$ there are no more intersections with $A_i$ by hypothesis. On the other hand it is clear that $\pos((\sfc_z \cap A)_j) \leq j < i$ for all $j<i$. This shows that $\dwidth_{v,w}(A_i)=1$.
    \vspace{1mm}
    
    \noindent The sets $A_i$, $i=1,\ldots,\dwidth_{v,w}(A)$, are disjoint and contained in $A$, so there exists one of them with $\mu(A_i) \leq \frac{\mu(A)}{\dwidth_{v,w}(A)}$. Setting $A' = A_i$ for such value of $i$ we get
    $$\dSR_{v,w}(A') = \mu(A_i) \le \frac{\mu(A)}{\dwidth_{v,w}(A)} = \dSR_{v,w}(A).$$
    As claimed at the beginning this is enough to conclude.
\end{proof}

\begin{remark}
    In the discrete argument the coarea inequality has been replaced by the easier mean estimate at the end of the proof above. Moreover the proof of Theorem \ref{thm:equivalence_minimum_energies_discrete} suggests that the infimum of the discrete separating ratio should be realized by 'slim' subsets. However notice that a discrete separating set with discrete width equal to $1$ is not necessarily `slim'. For instance one can consider the graph $V = \mathbb{Z}^2$ with $[(v,w)] \in E$ if and only if $v-w \in \{ (\pm 1,0), (0,\pm 1) \}$. Let $v:=(-5,0)$ and $w:=(5,0)$. Define $A:= \{ (x,y) \in V:\, |y|\ge |x|\}$. In such a case, $\dwidth_{v,w}(A)=1$ by simply considering the horizontal path, but the set is `thick'. 
\end{remark}

In Proposition \ref{prop:slim} we show that the infimum of the discrete separating ratio can be always computed among slim discrete separating sets, where for us 'slim' means that every point has position $1$. In other words a discrete separating set $A$ is slim if for every $z\in A$ there exists a path $\sfc \in \mathcal{P}_{v,w}$ such that $z = (\sfc \cap A)_1$. These sets resemble the boundaries of separating sets in the continuous case, that should be thought as model of slim sets. Let us first discuss the definition of slim sets. 
\begin{definition}[Slim separating sets]
    Let $(\V,E)$ be a graph, $v,w \in \V$. A subset $A\subseteq \V$ is called a slim separating set if
    \begin{itemize}
        \item[(a)] $A \in \dSS(v,w)$;
        \item[(b)] for every $z\in A$ there exists a path $\sfc \in \mathcal{P}_{v,w}$ such that $z = (\sfc \cap A)_1$.
    \end{itemize}
\end{definition}
Recall that the connected component of a point $v\in \V$ is the maximal subset $C$ of $\V$ satisfying $\mathcal{P}_{v,w} \neq \emptyset$ for every $w\in C$. A graph is connected if it has only one connected component. Observe that if $A \in \dSS(v,w)$ then also its intersection with the connected component containing $v$ and $w$ is a discrete separating set with same width and smaller measure. So it is not restrictive to suppose that a separating set is contained in the connected component of $v$ and $w$.
By definition a slim separating set is contained in the connected component of $v$ and $w$.

\begin{lemma}
\label{lemma:slim_equivalences}
Let $(\V,E)$ be a graph, $v,w \in \V$ and $A \in \dSS(v,w)$. Suppose $\mathcal{P}_{v,w} \neq \emptyset$ and that $A$ is contained in the connected component of $v$ and $w$. Then the following conditions are equivalent:
\begin{itemize}
    \item[(i)] $A$ is slim;
    \item[(ii)] the function $\pos_A$ restricted to $A$ is constantly equal to $1$;
    \item[(iii)] the function $\pos_A$ takes value in $\{0,1,+\infty\}$;
    \item[(iv)] for every point $z$ of the connected component of $v,w$ there exists a path $\sfc \in \mathcal{P}_{v,w}$ such that $z\in \sfc$ and $\#(\sfc \cap A) \in \{0,1\}$.
\end{itemize}
\end{lemma}

In particular condition (iv) says that $A$ is slim if and only if every point of the connected component of $v$ and $w$ lies inside a path joining $v$ to $w$ that intersects $A$ in exactly one point.

\begin{proof}
    Suppose $A$ is slim and let $z\in A$. By definition we get $\pos_A(z) \le 1$. On the other hand it is also clear from the definition that $\pos_A(z) \ge 1$ since $z\in A$. Therefore (ii) holds.\\
    Suppose now that (ii) holds and take $z\in \X$. If $z$ does not belong to the connected component of $v$ and $w$ then $\pos_A(z) = +\infty$ by definition. Otherwise there exists a path $\sfc = \{q_i\}_{i=0}^N \in \mathcal{P}_{v,w}$ containing $z$, say $z = q_{i_0}$. Consider the set $\{q_i\}_{i=0}^{i_0} \cap A$. If it is empty then $\pos_A(z) = 0$. Otherwise let $z'$ be the biggest element of $\{q_i\}_{i=0}^{i_0} \cap A$. By (ii) we can find a path $\sfc_{z'} \in \mathcal{P}_{v,w}$ such that $(\sfc_{z'} \cap A)_1 = z'$. Let us consider the concatenation $\tilde{\sfc}$ of $\sfc_{z'}$ up to $z'$ and $\sfc$ from $z'$. This is a path in $\mathcal{P}_{v,w}$ and, by construction, $\pos_{\tilde{\sfc},A}(z) = 1$, so $\pos_A(z) \le 1$. This shows (iii).\\
    We notice that (iv) is just a reformulation of (iii). It is also clear that (iv) implies (i), completing the equivalences. 
\end{proof}

\begin{proposition}
\label{prop:slim}
    Let $(\V, E, \mu)$ be a measure graph and let $v,w \in \V$. Suppose $\mathcal{P}_{v,w} \neq \emptyset$. Then:
    \begin{equation*}
        \inf_{A\subseteq \V}\dSR_{v,w}(A) = \inf_{A \textup{ slim}}\dSR_{v,w}(A).
    \end{equation*}
\end{proposition}

\begin{proof}
    It suffices to prove the inequality $\ge$. We recall that for sets $A \in \dSS(v,w)$ with $\dwidth_{v,w}(A)=1$ we have $\dSR(A) = \mu(A)$. Because of the proof of Theorem \ref{thm:equivalence_minimum_energies_discrete} and our assumption $\mathcal{P}_{v,w} \neq \emptyset$, it suffices to prove that for every $A \in \dSS(v,w)$ with $\dwidth_{v,w}(A)=1$, there exists $A'$ slim with $\mu(A') \le \mu(A)$. 
    Let $\pos_A$ be the position function associated to $A$ and take $A' := \{z\in \V \, : \, \pos_A(z) = 1 \}$. Then $A'$ belongs to $\dSS(v,w)$ as we already showed in the proof of Theorem \ref{thm:equivalence_minimum_energies_discrete}. It is also contained in the connected component of $v$ and $w$ since the position function is finite on $A'$. It is also slim since $\pos_{A'}(z) \le \pos_{A}(z) = 1$ for every $z\in A'$, while on the other hand we always have $\pos_{A'}(z) \ge 1$ for $z\in A'$. So $\pos_{A'}$ is constantly equal to $1$ on $A'$ and we can apply Lemma \ref{lemma:slim_equivalences}.
    Since it is a subset of $A$ it is clear that $\mu(A')\le \mu(A)$.
\end{proof}

\begin{remark}
\label{rem:gamma_convergence}
    Let $(\X,\sfd,\mm)$ be a doubling metric measure space. Let $x,y\in \X$ and let $(V_r, E_r, \mu_r)$ be the graph at scale $r>0$ introduced in Remark \ref{rmk:graph_approximation} with the requirement that $x,y\in\V_r$. It would be interesting to study the $\Gamma$-convergence of the functionals on graphs to the functionals in the continuous case. The goal would be to prove the following relations:
    $$\lim_{r\to 0} \inf_{A\subseteq \V_r} \frac{\dSR_{x,y}(A)}{r} \approx \inf_{A\subseteq \X\,\textup{closed}} \SR_{x,y}(A)$$
    and
    $$\lim_{r\to 0} \inf_{A\in \dSS(x,y, \V_r)} \frac{\mu_r(A)}{r} \approx \inf_{\Omega\in \SS_\textup{top}(x,y)} \mm^+(A).$$
    We do not know what are the conditions on $(\X,\sfd,\mm)$ to ensure this kind of estimates and we do not insist in this direction.
\end{remark}

\appendix
\section{Revision of conditions equivalent to the Poincaré inequality}
\label{sec:revision_Poincare}
There are many equivalent characterizations of the Poincaré inequality in terms of other geometric or analytical quantities, all involving curves connecting two points. However there is an asymmetry in these conditions: some of them are expressed in terms of $\Gamma_{x,y}$, while some of them in terms of  $\Gamma_{x,y}^L$ for a suitable $L 
< \infty$. The scope of this appendix is twofold: first we want to show that there is no difference in considering $\Gamma_{x,y}^L$ or $\Gamma_{x,y}$ in all of these conditions and secondly we will see that the equivalence is still valid for the pointwise version of the conditions, without requiring them \emph{for every couple of points}, assuming some regularity of the space. Furthermore we show that, given $p \in [1,\infty)$, the $p$-Poincaré inequality is equivalent to the existence of a $p$-pencil of curves, a result which is new for $p > 1$ (thus generalizing the previous works \cite{DurCarErikBiqueKorteShanmu21} and \cite{FasslerOrponen19}).\\
For the scope of this appendix the pointwise $p$-Poincaré inequality at $x,y\in \X$ is the following condition: there exist $C>0$ and $L \ge 1$ such that for every Borel function $u$ and every $g \in \UG(u)$ it holds
        \begin{equation}
        \label{eq:Riesz_PtPI_xy}
            |u(x)-u(y)|^p\le C \sfd(x,y)^{p-1} \int_\X g^p(z)\,\d \mm_{x,y}^{L}(z).
        \end{equation} 

\subsection{Keith's module characterization}

The first characterization is due to Keith (see \cite{Kei03}). In order to state it we need to recall the notion of $p$-modulus of a family of curves. Let $(\X,\sfd,\mm)$ be a metric measure space and let $p\ge 1$. Let $\Gamma$ be a family of rectifiable curves of $\X$. The set of admissible densities for $\Gamma$ is $${\rm Adm}(\Gamma):=\{ \rho \colon \X \to [0,+\infty) \text{ Borel, such that }\int_\gamma \rho\,\d s \ge 1 \text{ for all } \gamma \in \Gamma\}.$$
The $p$-modulus of the family $\Gamma$ with respect to the measure $\mm$ is
\begin{equation*}
   {\rm Mod}_p(\Gamma,\mm):=\inf \left\{ \int \rho^p\,\d \mm,\,:\, \rho \in {\rm Adm}(\Gamma) \right\}.
\end{equation*}

Keith's characterization reads as follows.

\begin{proposition}
\label{theo:PI_Keith}
    Let $(\X,\sfd,\mm)$ be a doubling metric measure space. Let $x,y \in \X$. Then the following conditions are quantitatively equivalent:
    \begin{itemize}
        \item[(i)] \eqref{eq:Riesz_PtPI_xy} holds;
        \item[(ii)] there exist $c >0$, $L\geq 1$ such that $\Mod_p(\Gamma_{x,y}, \mm_{x,y}^L) \geq c\sfd(x,y)^{1-p}$ for all $x,y\in \X$;
        \item[(iii)] there exist $c >0$, $L\geq 1$ such that $\Mod_p(\Gamma_{x,y}^L, \mm_{x,y}^L) \geq c\sfd(x,y)^{1-p}$ for all $x,y\in \X$.
    \end{itemize}
\end{proposition}
The equivalence between (i) and (ii) is essentially proved in \cite{Kei03}, we will write it in order to show that the pointwise version still works. The equivalence between (ii) and (iii) is consequence of the next lemma (see also the discussion at the beginning of the proof of \cite[Theorem 3.7]{DurCarErikBiqueKorteShanmu21}).
\begin{lemma}
    \label{prop:Mod_finito_infinito}
    Let $(\X,\sfd,\mm)$ be a $C_D$-doubling metric measure space, let $x,y\in \X$ and $L,\Lambda\geq 1$. Then $$\textup{Mod}_p(\Gamma_{x,y} \setminus \Gamma_{x,y}^\Lambda, \mm_{x,y}^L) \leq \frac{8}{\Lambda^p}C_D L \sfd(x,y)^{1-p}.$$
\end{lemma}
\begin{proof}
    The function $\rho = \frac{1}{\Lambda\sfd(x,y)}$ is admissible for the family $\Gamma_{x,y} \setminus \Gamma_{x,y}^\Lambda$, i.e.\ $\rho \in {\rm Adm}(\Gamma_{x,y} \setminus \Gamma_{x,y}^\Lambda)$. So
    $$\textup{Mod}_p(\Gamma_{x,y} \setminus \Gamma_{x,y}^\Lambda, \mm_{x,y}^L) \leq \int \rho^p \d\mm_{x,y}^L \leq \frac{1}{\Lambda^p} \frac{1}{\sfd(x,y)^p} \mm_{x,y}^L(\X) \leq \frac{8}{\Lambda^p}C_D L \sfd(x,y)^{1-p}.$$
    The last inequality follows from Lemma \ref{prop:properties_mxy}.
\end{proof}

\begin{proof}[Proof of Proposition \ref{theo:PI_Keith}]
    It is known that in every metric measure space the $p$-modulus of the family $\Gamma_{x,y}$ coincides with the $p$-capacity of the sets $\{x\}$ and $\{y\}$ (see \cite[Theorem 7.31]{Hei01}). Namely
    $${\rm Mod}_p(\Gamma_{x,y}, \mm_{x,y}^L) = \inf \left\{ \int \rho^p \,\d \mm_{x,y}^L \, : \, \rho \in \UG(u),\, u\colon \X \to \R, u(x) = 0, u(y) = 1 \right\}.$$
    For every $u\colon\X \to \R$ such that $u(x) = 0$ and $u(y) = 1$ we get, by \eqref{eq:Riesz_PtPI_xy},
    $$\int \rho^p\,\d\mm_{x,y}^L \ge \frac{1}{C}\sfd(x,y)^{1-p}.$$
    This shows that (i) implies (ii). \\
    We now show the converse. Let $u \colon \X \to \R$ and let $\rho \in \UG(u)$. If $u(x) = u(y)$ there is nothing to prove. Otherwise we define the function
    $$\bar{u}(z) = \left\vert \frac{u(z) - u(x)}{u(z) - u(y)}\right\vert$$
    and $\bar{\rho}(z) = \frac{\rho(z)}{\vert u(y) - u(x)\vert}$.
    By triangle inequality of the absolute value we obtain that $\bar{\rho} \in \UG(\bar{u})$. Since $\bar{u}(x) = 0$ and $\bar{u}(y) = 1$ we can use $\bar{\rho}$ in the estimate of the $p$-capacity. Namely we have
    $$c\,\sfd(x,y)^{1-p} \le {\rm Mod}_p(\Gamma_{x,y}, \mm_{x,y}^L) \le \int \bar{\rho}^p \,\d\mm_{x,y}^L = \frac{1}{\vert u(y) - u(x) \vert^p} \int \rho^p\,\d\mm_{x,y}^L,$$
    which is \eqref{eq:Riesz_PtPI_xy}.\\
    Finally, since it is clear that (iii) implies (ii), it is sufficient to prove that (ii) implies (iii). Suppose (ii) holds with constants $c,L$. Fix $\Lambda \geq 1$ such that $\frac{8}{\Lambda^p}C_DL \leq \frac{c}{2}$. Notice that $\Lambda$ depends only on $C_D,L,p$ and $c$. By Lemma \ref{prop:Mod_finito_infinito} we get $$\Mod_p(\Gamma_{x,y}^{\Lambda}, \mm_{x,y}^L) \geq \frac{c}{2} \sfd(x,y)^{1-p}.$$
    Setting $c' = \frac{c}{2}$ and $L' = \max\lbrace \Lambda, L \rbrace$ we get the thesis. Indeed
    $$\Mod_p(\Gamma_{x,y}^{L'}, \mm_{x,y}^{L'}) \geq \Mod_p(\Gamma_{x,y}^{\Lambda}, \mm_{x,y}^L) \geq c' \sfd(x,y)^{1-p}.$$
\end{proof}

\subsection{Pencil and $A_p$-connectedness characterizations}
We now move to other characterizations that are known in the literature with the name of $A_p$-connectedness (see \cite{ErikssonBique2019II}) and $p$-pencil (see for instance \cite{FasslerOrponen19} or \cite{DurCarErikBiqueKorteShanmu21} for the case $p=1$). We recall that given a complete and separable metric space $(\Y,\sfd_\Y)$ we denote by $\mathscr{P}(\Y)$ the set of Borel probability measures on $\Y$. We will always consider the topology induced by the weak convergence on $\mathscr{P}(\Y)$. We recall that a sequence of probability measures $\alpha_n \in \mathscr{P}(\Y)$ converges weakly to a probability measure $\alpha \in \mathscr{P}(\Y)$ if $\int u \,\d\alpha_n \to \int u \,\d\alpha$ for every bounded, continuous function $u\colon \Y \to \R$. This is the sequential convergence induced by a topology on $\mathscr{P}(\Y)$ that makes it a Polish space.
Moreover $\mathscr{P}(\Y)$ is compact if $\Y$ is. We recall also that every family of continuous curves with values in a metric space $\X$ is endowed with the supremum distance. In particular if $(\X,\sfd)$ is a complete metric space and $x,y\in \X$ are two points then the sets $\Gamma_{x,y}^L$ are compact for every $L\ge 1$ and $\Gamma_{x,y}$ is closed. Finally we denote by $\Lip_1(\X)$ the set of $1$-Lipschitz maps $u\colon \X \to \R$.

\begin{theorem}
    \label{thm:equivalence_pencil_Ap-connectedness}
    Let $(\X,\sfd,\mm)$ be a path connected, pointwise quasiconvex, doubling metric measure space. Let $x,y \in \X$. Then the following conditions are quantitatively equivalent.
    \begin{itemize}
        \item[(i)] \eqref{eq:Riesz_PtPI_xy} holds;
        \item[(ii)] there exist $C>0$ and $L\geq 1$ such that for every $g \geq 0$ Borel there exists $\gamma \in \Gamma_{x,y}^L$ such that 
        $$\left(\int_\gamma g \d s\right)^p \leq C\sfd(x,y)^{p-1}\int g^p \d\mm_{x,y}^L;$$
        \item[(iii)] there exist $C>0$ and $L\geq 1$ such that for every $g \ge 0$ bounded, $g \in \Lip_1(\X)$, there exists $\gamma \in \Gamma_{x,y}$ such that 
        $$\left(\int_\gamma g \d s\right)^p \leq C\sfd(x,y)^{p-1}\int g^p \d\mm_{x,y}^L.$$
        \item[(iv)] there exist $C >0$, $L\geq 1$ and $\alpha \in \mathscr{P}(\Gamma_{x,y})$ such that
        $$\left(\int \int_\gamma g \d\alpha\right)^p \leq C\sfd(x,y)^{p-1} \int g^p \d \mm_{x,y}^L$$
        for all $g \in \Lip_1(\X)$, $g\ge 0$ and bounded;
        \item[(v)] there exist $C >0$, $L\geq 1$ and $\alpha \in \mathscr{P}(\Gamma_{x,y}^L)$ such that
        $$\left(\int \int_\gamma g \d\alpha\right)^p \leq C\sfd(x,y)^{p-1} \int g^p \d \mm_{x,y}^L$$
        for all $g \geq 0$ Borel.
    \end{itemize}

    If moreover $(\X,\sfd)$ is locally $\Lambda$-quasiconvex then the conditions above are also equivalent to:
    \begin{itemize}
        \item[(vi)] there exist $C>0,L\ge 1$ such that \begin{equation}
        \label{eq:Riesz_PtPI_loclip}
            |u(x)-u(y)|^p\le C \sfd(x,y)^{p-1} \int_\X \left(\lip u(z)\right)^p\,\d \mm_{x,y}^{L}(z)
        \end{equation}
        for every $u\in \Lip(\X)$.
    \end{itemize}
\end{theorem}

Some connectivity assumption on $\X$ is necessary since the nice behaviour of every functions at just $x$ and $y$ as expressed by \eqref{eq:Riesz_PtPI_xy} is not enough to guarantee the existence of many curves in the space.\\
Condition (ii) is very similar to the $A_p$-connectedness condition of \cite{ErikssonBique2019II}. Of course we just ask it for a fixed couple of points. Condition (iii) is a weakening of (ii) since it requires the same condition for less functions and it allows to use more curves.\\
Conditions (iv) and (v) are closely related, at least in case $p=1$, to the so called pencil of curves. The formulation and the proof in case $p>1$ is new. The difference between these two conditions is similar to the difference between (ii) and (iii). Finally observe that in (ii) and (iii) it is required that for every function $g$ in a certain class there exists a curve, depending on the function $g$, for which the estimate holds. In conditions (iv) and (v) we ask for a class of curves on which the same condition holds \emph{in average for every function $g$} in the class. This step is made, as in \cite{DurCarErikBiqueKorteShanmu21}, with the help of the following well known min-max theorem.
\begin{proposition}[{\cite[Thm.\ 9.4.2]{Rudin80} (original proof in \cite{Sion58})}]
\label{prop:convex_optimization}
Let $\X$ be a vector space and $\Y$ be a topological vector space. Let $G \subseteq \X$ and $K \subseteq \Y$ be convex subsets, with $K$ compact. Let $F \colon G \times K \to \mathbb{R}$ be such that
\begin{itemize}
    \item[a)] $F(\cdot,y)$ is convex on $G$ for every $y \in K$;
    \item[b)] $F(x,\cdot)$ is concave and upper semicontinuous in $K$ for every $x \in G$.
\end{itemize}
Then
\begin{equation}
    \max_{y \in K} \inf_{x \in G} F(x,y) =  \inf_{x \in G} \max_{y \in K} F(x,y).
\end{equation}
\end{proposition}

\begin{proof}[Proof of Theorem \ref{thm:equivalence_pencil_Ap-connectedness}]
    We start assuming (i) and we show (iii). We tailor the argument in \cite[Thm.\ 1.5]{ErikssonBique2019II} to our situation. First of all we recall that $\X$ is rectifiable path connected by Lemma \ref{lemma:rectifiable_path}. 
    %In particular we can write $\X = \bigcup_{N \in \N} \X_N$, where $\X_N = \lbrace z \in \X \text{ s.t. } \Gamma_{x,z}^N \neq \emptyset \rbrace$ is closed by compactness of the set of $N$-quasigeodesics starting at $x$.
    For every $g \in \Lip_1(\X)$, $g \ge 0$ bounded, we define the map $u(z):= \inf_{\gamma \in \Gamma_{x,z}}\int_{\gamma} g\,\d s$. Observe that $u(z) < +\infty$ for every $z\in \X$ since $\X$ is rectifiable path connected. We claim that $u$ is continuous, hence Borel. Indeed let $z_n \in \X$ be points converging to $z$. Let $\varepsilon > 0$ and let $\gamma_n \in \Gamma_{x,z_n}$ be a curve such that $\int_{\gamma_n} g\,\d s \le u(z_n) + \varepsilon$. Let also $\Lambda_z,r_z$ be the constants of pointwise quasiconvexity at $z$. Since $\sfd(z_n,z) < r_z$ for $n$ big enough, then we can find curves $\eta_n \in \Gamma_{z_n,z}^{\Lambda_n}$ for every such $n$. The curve $\gamma_n \star \eta_n\in \Gamma_{x,z}$ satisfies
    $$u(z) \le \int_{\gamma_n \star \eta_n} g \,\d s \le u(z_n) + \varepsilon + \Lambda_z \sfd(z_n,z) \Vert g \Vert_\infty.$$
    For $n \to +\infty$ and $\varepsilon \to 0$ we get $u(z) \le \limi_{n\to +\infty} u(z_n)$. Arguing similarly we get also the other inequality, namely $u(z) \ge \lims_{n\to +\infty} u(z_n)$. Thus $u$ is Borel and $g$ is clearly an upper gradient of $u$. By (i) and the definition of $u$ we find $\gamma \in \Gamma_{x,y}$ such that
    $$\left(\int_\gamma g \d s\right)^p \leq C\sfd(x,y)^{p-1} \int g^p(z)\d\mm_{x,y}^L.$$
    The next step is to prove that (iii) implies (v). We fix $n \in \N$ and we observe that $\Lip_1(\X, [0,n])$ is a convex subset of the vector space $(C^0(\X,\R), \Vert \cdot \Vert_\infty)$. Moreover $\mathscr{P}(\Gamma_{x,y}^\Lambda)$ is a compact, convex subset of the vector space of all measures of $\Gamma_{x,y}^\Lambda$ endowed with the weak convergence.
    For every $\Lambda \geq 1$ we consider the functional:
    $$F_\Lambda^n\colon \mathcal{P}(\Gamma_{x,y}^\Lambda) \times \Lip_1\left(\X, \left[0,n\right]\right), \quad (\alpha,g) \mapsto C^{1/p}\sfd(x,y)^{\frac{p-1}{p}} \left(\int g^p \d\mm_{x,y}^L\right)^{1/p} - \int \int_\gamma g \d \alpha.$$
    Observe that the functional has always the same expression, the only dependence on $\Lambda$ and $n$ is in the domain of definition. 
    We notice also that $F_\Lambda^n(\cdot, g)$ is concave and upper semicontinuous because it is linear, while $F_\Lambda^n(\alpha, \cdot)$ is convex because it is a $L^p$-norm plus a linear term. By applying Proposition \ref{prop:convex_optimization} we have
    $$M_\Lambda^n:= \max_{\alpha \in \mathscr{P}(\Gamma_{x,y}^\Lambda)} \inf_{g\in \Lip_1(\X, [0, n])} F_\Lambda^n(\alpha,g) = \inf_{g\in \Lip_1(\X, [0, n])}\max_{\alpha \in \mathscr{P}(\Gamma_{x,y}^\Lambda)} F_\Lambda^n(\alpha,g)$$
    for every $\Lambda\geq 1$. We claim that $\limi_{\Lambda\to +\infty}M_\Lambda^n \geq 0$, for $n$ fixed. Suppose it is not the case. Then we can find $\varepsilon > 0$, a sequence $\Lambda_j \to +\infty$ and a sequence of functions $g_j \in \Lip_1(\X, [0,n])$ such that for every $\alpha \in \mathscr{P}(\Gamma_{x,y}^{\Lambda_j})$ it holds $F_{\Lambda_j}^n(\alpha, g_j) \leq -\varepsilon < 0$. By Ascolì-Arzela we can suppose that $g_j$ converges uniformly to a limit function $g_\infty \in \Lip_1(\X, [0,n])$. By assumption (iii) we can find $\gamma \in \Gamma_{x,y}$ such that 
    $$\left( \int_\gamma g_\infty \d s \right)^p \leq C\sfd(x,y)^{p-1} \int g_\infty^p \d\mm_{x,y}^L.$$
    By definition we have that $\gamma \in \Gamma_{x,y}^{\Lambda_j}$ if $j\ge j_0$, for some $j_0$.
    Therefore, denoting by $\delta_\gamma \in \mathscr{P}(\Gamma_{x,y})$ the atomic measure concentrated on $\gamma$, we get $F_{\Lambda_{j_0}}^n(\delta_{\gamma}, g_\infty) \geq 0$. Since $\gamma$ has finite length and $\mm_{x,y}^L$ is finite and since $g_j$ converges uniformly on compact subsets to $g_\infty$, we have that $F_{\Lambda_j}^n(\delta_\gamma, g_j) = F_{\Lambda_{j_0}}^n(\delta_\gamma, g_j)$ converges to $F_{\Lambda_{j_0}}^n(\delta_\gamma, g_\infty)$, which is a contradiction. 
    Therefore $\limi_{\Lambda\to +\infty}M_\Lambda^n \geq 0$. This means that for all $\varepsilon > 0$ there exist $\Lambda_\varepsilon^n \geq 1$ and $\alpha_\varepsilon^n \in \mathscr{P}(\Gamma_{x,y}^{\Lambda_\varepsilon^n})$ such that $F_{\Lambda_\varepsilon^n}^n(\alpha_\varepsilon^n, g) \geq -\varepsilon$ for all $g\in \Lip_1(\X, [0,n])$. We consider only $\varepsilon < (8CC_DL)^{\frac{1}{p}}\sfd(x,y)$.
    We claim that if $\Lambda_0 = 4(8C C_D L)^{\frac{1}{p}} $ then $\alpha_\varepsilon^n(\Gamma_{x,y}^{\Lambda_0}) \geq \frac{1}{2}$. Indeed if not then $\alpha(\Gamma_{x,y} \setminus \Gamma_{x,y}^{\Lambda_0}) > \frac{1}{2}$. We apply the condition above to $g=1$ getting $F^n_{\Lambda^n_\varepsilon}(\alpha^n_\varepsilon, 1) \geq - \varepsilon$, which means
    $$C^{\frac{1}{p}}\sfd(x,y)^{\frac{p-1}{p}} \mm_{x,y}^L(\X)^{\frac{1}{p}} - \int \ell(\gamma) \d \alpha_\varepsilon^n (\gamma) \geq -\varepsilon.$$
    This implies
    $$\frac{\Lambda_0}{2}\sfd(x,y) < \int\ell(\gamma) \d \alpha_\varepsilon^n(\gamma) \leq C^{\frac{1}{p}}\sfd(x,y)^{\frac{p-1}{p}}\mm_{x,y}^L(\X)^{\frac{1}{p}} + \varepsilon \leq 2(8CC_DL)^{\frac{1}{p}} \sfd(x,y),$$
    which is a contradiction. In the last inequality we used Lemma \ref{prop:properties_mxy}.
    Renormalizing the restriction of $\alpha_\varepsilon^n$ to $\Gamma_{x,y}^{\Lambda_0}$ we obtain the following. For every $\varepsilon > 0$ define $\beta_\varepsilon^n:= \alpha_\varepsilon^n(\Gamma_{x,y}^{\Lambda_0})^{-1}\,\alpha_\varepsilon^n \restr{\Gamma_{x,y}^{\Lambda_0}} \in \mathscr{P}(\Gamma_{x,y}^{\Lambda_0})$ and 
    $$\left(\int\int_\gamma g\,\d s \,\d\beta_\varepsilon^n\right)^p \leq 2^pC\sfd(x,y)^{p-1}\int g^p \d\mm_{x,y}^L + \varepsilon$$
    for all $g\in \Lip_1(\X, [0,n])$.
    The family $\{\beta_\varepsilon^n \}_{\varepsilon > 0}$ is relatively compact. Hence we take a limit point $\beta^n$ of the sequence $\beta_\varepsilon^n$ for $\varepsilon$ going to zero. We have that $\beta^n \in \mathscr{P}(\Gamma_{x,y}^{\Lambda_0})$
    %Since $\Gamma_{x,y}^{\Lambda_0}$ is closed, we get that $\beta^n \in \mathscr{P}(\Gamma_{x,y}^{\Lambda_0})$. Indeed, since $\beta^n_\varepsilon \rightharpoonup \beta^n_\varepsilon$, $1=\limsup_{\varepsilon \to 0} \beta^n_\varepsilon(\Gamma_{x,y}^{R_0}) \le  \beta^n(\Gamma_{x,y}^{R_0}) \le 1$ 
    and 
    $$\left(\int\int_\gamma g\,\d s\, \d\beta^n\right)^p \leq 2^pC\sfd(x,y)^{p-1}\int g^p \d\mm_{x,y}^L$$
    for all $g\in \Lip_1(\X, [0,n])$, by the properties of weak convergence, since the map $\Gamma_{x,y}^{\Lambda_0}\ni \gamma \mapsto \int_{\gamma} g\,\d s \in [0,\infty)$ is bounded and lower semicontinuous. 
    We now take a limit point $\beta \in \mathscr{P}(\Gamma_{x,y}^{\Lambda_0})$ of the sequence $\beta^n$ for $n$ going to $+\infty$. Then we have that 
    $$\left(\int\int_\gamma \min\lbrace g, M \rbrace\,\d s\, \d\beta\right)^p \leq 2^pC\sfd(x,y)^{p-1}\int \min\lbrace g, M \rbrace^p \d\mm_{x,y}^L$$
    for all $g\in \Lip_1(\X, [0,+\infty))$ and all $M \in \N$,
    again by weak convergence. 
    %By Fatou's lemma on the left hand side and dominate convergence on the right hand side we get
    %\begin{equation}
    %\left(\int\int_\gamma \min\lbrace g, M \rbrace\,\d s\, \d\beta\right)^p \leq 2^pC\sfd(x,y)^{p-1}\int \min\lbrace g, M \rbrace^p \d\mm_{x,y}^L
    %\end{equation}
    %for all $g\in \Lip_1(\X, [0,+\infty))$ and all $M \in \N$ {\color{red}Che cosa e' cambiato?}. 
    The right hand side measure is concentrated on a compact subset of $\X$, while the left hand side one is concentrated of curves of bounded length, so whose image lives in a compact subset of $\X$. This means that for a fixed $g\in \Lip_1(\X,[0,+\infty))$ we have $\min\lbrace g, M \rbrace = g$ if $M$ is big enough. Therefore we get
    
    \begin{equation}
    \label{eq:pencil_estimate_inside_equivalence}
    \left(\int\int_\gamma g\,\d s\, \d\beta\right)^p \leq 2^pC\sfd(x,y)^{p-1}\int g^p \d\mm_{x,y}^L
    \end{equation}
    for all $g\in \Lip_1(\X, [0,+\infty))$.
    So the same estimate holds for all $g\in \Lip(\X, [0,+\infty))$. A standard approximation argument (see \cite[Corollary 4.2.3 and the theorem below]{HKST15}) gives that \eqref{eq:pencil_estimate_inside_equivalence} holds for all $g \geq 0$ Borel. In particular, if we set $C' = 2^pC$ and $L' = \Lambda_0$ we get
    $$\left(\int\int_\gamma g\,\d s\, \d\beta\right)^p \leq C'\sfd(x,y)^{p-1}\int g^p \d\mm_{x,y}^{L'}$$
    for all $g \geq 0$ Borel, and $\beta \in \mathscr{P}(\Gamma_{x,y}^{L'})$.
    This is property (v).\\
    It is clear that (v) implies (iv), that (v) implies (ii) and that (ii) implies (iii). Moreover (iv) implies (iii), so (ii)-(v) are equivalent. To conclude the cycle of equivalences we need to show that (v) implies (i) estimating the $p$-modulus of $\Gamma_{x,y}$. For a given $\rho$ which is admissible for the definition of  $\textup{Mod}_p(\Gamma_{x,y},\mm_{x,y}^L)$ we get
    $$1 \leq \left(\int \int_\gamma \rho\,\d s\,\d \alpha \right)^p \leq C \sfd(x,y)^{p-1}\int \rho^p\d\mm_{x,y}^L,$$
    so $\textup{Mod}_p(\Gamma_{x,y},\mm_{x,y}^L)\geq \frac{1}{C}\sfd(x,y)^{1-p}$, which implies (i) by Proposition \ref{theo:PI_Keith}.\\
    Assume now that $\X$ is locally quasiconvex. It is always true that if $u\in \Lip(\X)$ then $\lip u \in \UG(u)$. So if (i) holds then (vi) holds as well. On the other hand we adapt the proof of (i) implies (iii) to show that (vi) implies (iii). Let $g \in \Lip_1(\X)$, $g\ge 0$ bounded and define $u(z) = \inf_{\gamma \in \Gamma_{x,z}} \int_\gamma g \,\d s$. Let $r_z$ be such that every two points belonging to $B_{r_z}(z)$ can be joined by a $\Lambda$-quasigeodesic. Arguing exactly as above we deduce that $u$ is $(\Lambda \Vert g \Vert_\infty)$-Lipschitz on $B_{r_z}(z)$. Therefore $u\in \Lip_\loc(\X)$. By Lemma \ref{lemma:loclip_to_lip_up_to_scale} we have that the restriction of $u$ to $\overline{B}_{x,y}^L$ is Lipschitz. By McShane Extension Theorem we can find a map $\tilde{u} \in \Lip(\X)$ that coincides with $u$ on $\overline{B}_{x,y}^L$. Applying (vi) to $\tilde{u}$ we get
    $$\vert u(x) - u(y) \vert ^p = \vert \tilde{u}(x) - \tilde{u}(y)  \vert ^p \le C\sfd(x,y)^{p-1}\int_\X (\lip\tilde{u})^p\,\d\mm_{x,y}^L = C\sfd(x,y)^{p-1}\int_\X (\lip u)^p\,\d\mm_{x,y}^L,$$
    where the last equality holds because $\mm_{x,y}^L$ is concentrated on the open set $B_{x,y}^L$ where $u$ and $\tilde{u}$ coincide. The missing step is the following estimate: we claim that $\lip u \le \Lambda g$. Indeed let $z \in \X$ and take $0<r<r_z$. For every $w\in B_{r}(z)$ we can find a curve $\gamma \in \Gamma_{z,w}^\Lambda$. By definition of upper gradient we have
    $$\vert u(z) - u(w) \vert \le \int_\gamma g \le \Lambda \sfd(z,w) \max_{B_{r}(z)} g.$$
    Since $g$ is continuous we get
    $$g(z) = \lim_{r\to 0}\max_{B_{r}(z)} g \ge \frac{1}{\Lambda} \lim_{r\to 0} \sup_{w\in B_r(z)}\frac{\vert u(z) - u(w) \vert}{\sfd(z,w)} = \lip u (z).$$
    Combining this fact with the estimate above we conclude that
    $$\vert u(x) - u(y) \vert ^p \le \frac{C}{\Lambda^p} \sfd(x,y)^{p-1} \int_\X g\,\d \mm_{x,y}^L.$$
    Then by definition of $u$ we find $\gamma \in \Gamma_{x,y}$ such that
    $$\left(\int_\gamma g \d s\right)^p \leq \frac{C}{\Lambda^p}\sfd(x,y)^{p-1} \int g^p(z)\d\mm_{x,y}^L,$$
    i.e. (iii) holds.
\end{proof}

\bibliographystyle{abbrv}
\bibliography{biblio.bib}
\end{document}